\documentclass[a4paper,10pt,fleqn]{amsart}
\pdfoutput=1

 \usepackage[paperwidth=17cm,%
     paperheight=24cm,%
     footskip=0pt,%
     headsep=0.5cm,%
     headheight=0.5cm,%
     width=12.2cm,%
     left=2cm,%
     top=2cm,%
     bottom=2.5cm]{geometry}

\usepackage[utf8]{inputenc}
\usepackage[T1]{fontenc}

\usepackage{todonotes}
\usepackage{microtype}

\usepackage[english]{babel}
\usepackage{csquotes} %

\usepackage{setspace}
\usepackage{amsmath,amsthm,amsfonts,amssymb}
\usepackage{mathtools}

\usepackage[shortlabels]{enumitem}

\usepackage{xcolor}
\usepackage{xspace}
\usepackage[hypertexnames=false]{hyperref}

\definecolor{darkblue}{rgb}{0,0,0.5}
\definecolor{cerule}{RGB}{53,122,183}
\definecolor{cardinal}{RGB}{184,32,16}
\hypersetup{
    colorlinks,
    linkcolor={black},
    citecolor={cerule},
    urlcolor={cardinal}
}

\usepackage{float}
\usepackage[]{algpseudocode}
\algrenewcommand\textproc{}

\usepackage{tikz}
\usetikzlibrary{decorations,arrows,automata,trees,fit,shapes,intersections,positioning}
\definecolor{newLightBrown}{RGB}{230,51,18}

\usepackage[ruled]{caption}      %
\usepackage{subcaption}

\floatstyle{ruled}
\newfloat{algo}{tp}{lop}
\floatname{algo}{Algorithm}
\allowdisplaybreaks[2]

\usepackage[
  citestyle=authoryear-icomp,
  bibstyle=authoryear,
  backend=biber,
  isbn=false,
  maxcitenames=2,
  maxbibnames=6,
  giveninits=true,
  sortcites=true,
  url=false
]{biblatex}
\bibliography{periodgap.bib}

\AtEveryBibitem{\clearfield{month}}
\AtEveryBibitem{\clearfield{day}}

\DeclareNameAlias{sortname}{first-last}

\def\eqdef{\stackrel{.}{=}}
\newcommand{\st}{\mathrel{}\middle|\mathrel{}}

\usepackage{mathabx}
\def\VERT{\vvvert}
\def\tn#1{\left\VERT #1 \right\VERT}

\def\epsilon{\varepsilon}
\def\ud{\mathrm{d}}

\def\mop#1{\operatorname{#1}}

\def\abs#1{\left| #1 \right|}

\def\MP{\mathbb{P}^{\mathbf{n}}}
\def\bfn{\mathbf{n}}

\def\Pic{\operatorname{Pic}}
\def\Hii{H^{1,1}}
\def\NL{\mathcal{N\!L}}
\def\wNL{\widetilde{\mathcal{N\!L}}}
\def\pNL{\mathrm{NL}}

\def\Qbar{\overline{\mathbb{Q}}}

\def\C{\mathbb{C}}
\def\Z{\mathbb{Z}}

\def\ppp{\mathbb{P}^3}

\def\P{\mathbb{P}}

\def\ca{\mathcal{A}}
\def\cb{\mathcal{B}}

\def\ch{\mathcal{H}}

\def\cp{\mathcal{P}}

\def\bx{\mathbf{x}}

\def\Xf{X_{\! f}}
\def\Cf{C_{\! f}}
\def\epsf{\epsilon_{\! f}}

\def\volproj{\mop{Vol}}

\def\bsmale{\beta_{\mathrm{Smale}}}
\def\gsmale{\gamma_{\mathrm{Smale}}}

\def\mult{\mop{mult}}

\newcommand{\norm}[1]{\left\| #1 \right\|}

\title{Separation of periods of quartic surfaces}
\newtheorem{theorem}{Theorem}%
\newtheorem{proposition}[theorem]{Proposition}
\newtheorem{lemma}[theorem]{Lemma}
\newtheorem{corollary}[theorem]{Corollary}
\newtheorem{conjecture}[theorem]{Conjecture}

\theoremstyle{definition}

\theoremstyle{remark}

\makeatletter
\@namedef{subjclassname@2020}{%
  \textup{2020} Mathematics Subject Classification}
\makeatother

\usepackage{color}

\begin{document}

\author[P.~Lairez]{Pierre Lairez}
\address{Pierre Lairez, Inria, France}
\email{pierre.lairez@inria.fr}
\urladdr{pierre.lairez.fr}

\author[E.~C.~Sert\"oz]{Emre Can Sert\"oz}
\address{Emre Can Sert\"oz, Leibniz University Hannover, Hannover, Germany}
\email{emre@sertoz.com}
\urladdr{emresertoz.com}

\subjclass[2020]{14Q10, 14J28, 32G20, 11Y16, 14Q20, 11J99}
\keywords{K3 surfaces, periods, Diophantine approximation, Hodge loci, effective mathematics}

\date{\today}

\begin{abstract}
We give a computable lower bound on the distance between two distinct periods of a given quartic surface defined over the algebraic numbers.
The main ingredient is the determination of height bounds on components of the Noether--Lefschetz loci. 
This makes it possible to study the Diophantine properties of periods of quartic surfaces and to certify a part of the numerical computation of their Picard groups.
\end{abstract}  

\maketitle

\setcounter{tocdepth}{2}

\section{Introduction}
\label{sec:introduction}

Periods are a countable set of complex numbers 
containing all the algebraic numbers as well as many of the transcendental constants of nature. 
In light of the ubiquity of periods in mathematics and the sciences, \textcite{KontsevichZagier_2001} 
ask for the development of an algorithm to check for the equality of two given periods. 
We solve this problem for periods coming from quartic surfaces by 
giving a computable separation bound, that is, 
a lower bound on the minimum distance between distinct periods.

Let $f \in \mathbb{\C}[w,x,y,z]_4$ be a homogeneous quartic polynomial  
defining a smooth quartic $\Xf$ in~$\mathbb{P}^3(\mathbb{C})$.
The periods of $\Xf$ are the integrals of a nowhere vanishing holomorphic $2$-form on $\Xf$ over
integral $2$-cycles in $\Xf$. The periods can also be given in the form of integrals of a rational function
\begin{equation}\label{eq:intro}
  \frac{1}{2\pi i} \oint\nolimits_\gamma \frac{\ud x\,\ud y\,\ud z}{f(1, x, y, z)},
\end{equation}
where~$\gamma$ is a $3$-cycle in~$\mathbb{C}^3 \setminus \Xf$. The
integral~\eqref{eq:intro} depends only on the homology class of~$\gamma$. These periods form a 
group under addition. 
The geometry of quartic surfaces dictates that there are only $21$~independent $3$-cycles 
in~$\mathbb{C}^3 \setminus \Xf$. These give $21$~periods $\alpha_1,\dotsc,\alpha_{21} \in \mathbb{C}$ such that
the integral over any other $3$-cycle is an integer linear combination of these periods. 

It is possible to compute the periods to high precision \parencite{Sertoz_2019}, 
typically to thousands of decimal digits,
and to deduce from them interesting algebraic invariants 
such as the Picard group of $\Xf$~\parencite{LairezSertoz_2019}. 
This point of view has been fruitful for computing
algebraic invariants for algebraic curves from their periods
\parencite{VanWamelen_1999,CostaMascotSijslingVoight_2019,BruinSijslingZotine_2019,BookerSijslingSutherlandVoightYasaki_2016}.

For quartic surfaces, the computation of the Picard group reduces to
computing the lattice in $\Z^{21}$ of integer relations 
$x_1 \alpha_1 + \dotsb + x_{21}\alpha_{21} = 0$, $x_i \in \Z$. A 
basis for this lattice can be guessed from approximate $\alpha_i$'s using lattice reduction algorithms.  
But is it possible to \emph{prove} that all guessed relations are true relations? 
Previous work related to this question \parencite{Simpson_2008} require explicit construction 
of algebraic curves on~$\Xf$, which becomes challenging very quickly.  
Instead, we give a method of proving relations by checking them at a predetermined finite precision.
At the moment, this is equally challenging, but we conjecture that the numerical approach
can be made asymptotically faster, see \S \ref{sec:comp-compl} for details.

The Lefschetz theorem on $(1,1)$-classes (\S\ref{sec:lefschetz}) associates a divisor on $\Xf$ to any integer relation between the periods of~$\Xf$. In turn, the presence of a divisor %
imposes algebraic conditions on the coefficients of $f$. 
Such algebraic conditions define the \emph{Noether--Lefschetz loci} on the space of quartic polynomials (\S\ref{sec:NL}). 
In addition to the degree computations of \textcite{MaulikPandharipande_2013}, we give height bounds on the polynomial equations defining the Noether--Lefschetz loci (Theorem~\ref{thm:nl-bounds}). 
These lead to our main result (Theorem~\ref{thm:separation-bound}): 
Assume $f$ has algebraic coefficients, then for $x_i \in \Z$,
\begin{equation}\label{eq:main_result}
  \tikz[baseline,anchor=base]{
    \node [draw=cardinal,inner sep=.5em, thick, rounded corners=.2em] {
      $x_1 \alpha_1 + \dotsb + x_{21}\alpha_{21} = 0 \text{ or } \abs{x_1 \alpha_1 + \dotsb + x_{21}\alpha_{21}}
      > 2^{-c^{\max_i \abs{x_i}^9}} \mathrlap{\phantom{f_{f_{f}}}}
      $
    }}
\end{equation}
for some constant~$c > 0$ depending only on~$f$ and the choice of the 21 independent 3-cycles (see Theorem~\ref{thm:separation-bound} for a coordinate-free formulation).
The constant~$c$ is computable in rather simple terms and without prior knowledge of the Picard group of~$\Xf$.
We use the term ``computable'' in the sense of ``computable with a Turing machine'', not ``primitive recursive'',
as our suggested algorithm to compute~$c$ depends, through Lemma~\ref{lem:technical1}, on the numerical computation of a nonzero constant (depending on~$f$),
whose magnitude is not known \emph{a priori}, only the fact that it is nonzero.

The expression~\eqref{eq:main_result} is essentially a lower bound for the linear independence measure~\parencite[Ch.11]{Shidlovskii1989} for the periods of $\Xf$. Our construction of this bound bears a loose resemblence to the ideas involved in the statement of the analytic subgroup theorem \parencite{Wustholz1989}, and in particular, to the Wüstholz--Masser period theorem~\parencite{Masser1993}. We briefly comment on this analogy in~\S\ref{sec:analogy}.

As a consequence of the separation bound~\eqref{eq:main_result}, we apply a construction in the manner of \textcite{Liouville_1851}
and prove, for instance, that the number 
\begin{equation}
\sum_{n\geq 0} (2 \upuparrows 3n)^{-1}
\end{equation}
is not a quotient of two periods of a single quartic surface defined over~$\Qbar$,
where~$2 \upuparrows 3n$ denotes an exponentiation tower with~$3n$ twos (Theorem~\ref{thm:liouville}, with~$\theta_{i+1} = 2^{2^{2^{\theta_i}}}$).

The methods we employed to attain the period separation bound~\eqref{eq:main_result} can in principle be generalized to separate the periods of some other algebraic varieties, e.g., of cubic fourfolds. We discuss these and other generalizations in~\S\ref{sec:conclusion}.

\subsection*{Acknowledgements}
We thank Bjorn Poonen for suggesting the use of heights of Noether--Lefschetz loci 
and Gavril Farkas for suggesting the paper by Maulik and Pandharipande.  
We also thank Alin Bostan and Matthias Schütt for numerous helpful comments.
We thank the referee for a careful reading of the paper.

ECS was supported by Max Planck Institute for Mathematics in the Sciences, Leibniz University Hannover, and Max Planck Institute for Mathematics. 
PL was supported by the project \emph{De Rerum Natura} ANR-19-CE40-0018 of the French National Research Agency (ANR).

\section{Periods and deformations}

\subsection{Construction of the period map}

For any non-zero homogeneous polynomial $f$ in $\mathbb{C}[w,x,y,z]$, let~$\Xf$
denote the surface in~$\mathbb{P}^3$ defined as the zero locus of $f$. 
Let~$R \eqdef \mathbb{C}[w,x,y,z]$ and let~$R_4 \subset R$ be the subspace of degree~$4$ homogeneous polynomials.
Let~$U_4 \subset R_4$ denote the dense open subset of all homogeneous polynomials~$f$ of degree~$4$ 
such that~$\Xf$ is smooth. %
For our purposes, it will be useful to consider not only the periods of a single quartic
surface~$\Xf$ but also the \emph{period map} to study the dependence of periods on $f$.

The topology of~$\Xf$ does not depend on~$f$ as long as~$\Xf$ is smooth: given
two polynomials~$f$ and~$g \in U_4$, we can connect them by a continuous path
in~$U_4$ and the surface~$\Xf$ deforms continuously along this path, giving a
homeomorphism~$\Xf\simeq X_g$, which is uniquely defined up to isotopy. In
particular, if we fix a base point~$b \in U_4$, then for
every~$f \in \widetilde{U}_4$, where~$\widetilde{U}_4$ is a universal covering of~$U_4$,
we have a uniquely determined isomorphism of cohomology groups 
$H^2(X_b, \mathbb{Z}) \simeq H^2(\Xf, \mathbb{Z})$. Let~$H_\mathbb{Z}$ denote
the second cohomology group of $X_b$, which is isomorphic to~$\mathbb{Z}^{22}$ \parencite[e.g.][\S 1.3.3]{Huybrechts_2016}. 

The hyperplane class and its multiples are redundant for the problem we are interested as their periods are $0$. In practice, therefore, we work with a rank $21$ quotient lattice. The map~\eqref{eq:tube} below identifies this quotient with the cohomology of the complement of the quartic.

An element of $\widetilde{U}_4$ determines a polynomial
$f \in U_4$ together with an identification of $H^2(\Xf,\Z)$ with $H_\Z$.
We often work locally around a given polynomial~$f$ and, in that case, we do not
actively distinguish between~$U_4$ and its universal covering.

The group~$H_\mathbb{Z}$ is endowed with an even unimodular pairing 
\begin{equation}
(x,y) \in H_\mathbb{Z} \times H_\mathbb{Z} \to x\cdot y \in \mathbb{Z},
\end{equation}
given by the intersection form on cohomology. Through this pairing, the second homology and cohomology groups are canonically identified with one another. 
For K3 surfaces, such as smooth quartic surfaces in $\ppp$, the structure of the lattice~$H_\mathbb{Z}$ with its intersection form is explicitly known \parencite[Proposition~\textbf{1}.3.5]{Huybrechts_2016}.
The fundamental class of a generic hyperplane section of~$\Xf$ gives an element of~$H_\mathbb{Z}$ denoted by~$h$.

Furthermore, the complex cohomology group~$H^2(\Xf, \mathbb{C})$, which is just~$H_\mathbb{C} \eqdef H_\mathbb{Z} \otimes \mathbb{C}$, is isomorphic to the corresponding de~Rham cohomology~$H^2_\text{dR}(\Xf, \mathbb{C})$ group as follows.
Elements of~$H^2_\text{dR}(\Xf, \mathbb{C})$ are represented by differential $2$-forms. To a form~$\Omega$
one associates the element~$\Theta(\Omega)$ of~$H^2(\Xf, \mathbb{C})$ given by the map
\begin{equation}\label{eq:20}
  \Theta(\Omega) \colon [ \gamma ] \in H_ 2(\Xf, \mathbb{C}) \mapsto \int_\gamma \Omega \in \mathbb{C}.
\end{equation}

The group $H^2_\text{dR}(\Xf, \mathbb{C})$ has a distinguished element~$\Omega_f$, a nowhere vanishing holomorphic $2$-form, described below.
Every other holomorphic $2$-form on~$\Xf$ is a scalar multiple of $\Omega_f$ \parencite[Example~\textbf{1}.1.3]{Huybrechts_2016}.
Mapping $\Omega_f$ to~$H_\mathbb{C}$ gives rise to the \emph{period map}
\begin{equation}\label{eq:32}
  \mathcal{P} \colon f \in \widetilde{U}_4 \mapsto \omega_f \eqdef \Theta(\Omega_f) \in H_\mathbb{C}.
\end{equation}
The coordinates of the \emph{period vector}~$\omega_f$, in some fixed basis of~$H_\mathbb{Z}$,
generates the group of periods of~$\Xf$.

There is a standard Thom--Gysin type map in homology 
\begin{equation}\label{eq:tube}
  T \colon H_2(\Xf, \mathbb{Z}) \to H_3(\mathbb{P}^3 \setminus \Xf, \mathbb{Z}),
\end{equation}
see \parencite[p.159]{Voisin_2003} for a modern description.
Roughly speaking, $T$ takes the class of a $2$-cycle in $\Xf$ and returns the class of a narrow $S^1$-bundle around the cycle lying entirely in $\mathbb{P}^3 \setminus \Xf$. See \parencite[\S 3]{Griffiths_1969} for this classical interpretation. The map~$T$ is a surjective morphism and its kernel is generated by the class of a hyperplane section of~$\Xf$.

We choose $\Omega_f$ so that the following identity holds
\begin{equation}\label{eq:31}
  \int_\gamma \Omega_f = \frac{1}{2\pi i} \int_{T(\gamma)} \frac{\ud x\,\ud y\, \ud z}{f(1, x, y, z)}.
\end{equation}
Therefore, in view of~\eqref{eq:20}, the coefficients of~$\omega_f$ in a basis of~$H_\mathbb{Z}$
coincides with periods as defined in~\eqref{eq:intro}.

The image~$\mathcal{D}$ of the period map~$\mathcal{P}$ is called the \emph{period domain}.
It admits a simple description:
\begin{equation}\label{eq:42}
  \mathcal{D} \eqdef \mathcal{P}(\widetilde{U}_4) = \left\{ w \in H_\mathbb{C} \setminus \left\{ 0 \right\}\st w\cdot h =0, w\cdot w = 0, w\cdot \overline w > 0 \right\},
\end{equation}
where~``$\cdot$'' denotes the intersection form on~$H_\mathbb{Z}$, extended to~$H_\C$ by $\C$-linearity, and~$h$ the fundamental class of a hyperplane section,
as introduced above \parencite[Chapter~6]{Huybrechts_2016}.
Moreover, by the local Torelli theorem for K3 surfaces \parencite[Proposition~\textbf{6}.2.8]{Huybrechts_2016}, the map~$\mathcal{P}$ is a submersion: its derivative at any point of~$\widetilde{U}_4$ is surjective.

\subsection{The Lefschetz (1,1)-theorem}\label{sec:lefschetz}

Lefschetz proved that the linear integer relations between the periods of a quartic surface~$\Xf$ are in correspondence with
homology classes coming from algebraic curves in~$\Xf$. We now explain this statement in more detail.
Let~$C \subset \Xf$ be an algebraic curve. Its fundamental class is the element~$[C]$
of~$H_\mathbb{Z}$ obtained as the Poincaré dual of the homology class of~$C$. Here we identify $H_\mathbb{Z}$ with~$H^2(X_f, \Z)$ by fixing a preimage of~$f$ in $\widetilde{U}_4$.
The Picard group~$\Pic(\Xf)$ of~$\Xf$ is
the sublattice of~$H_\mathbb{Z}$ spanned by the fundamental classes of algebraic
curves.

It follows from the definition that for any class~$[\Omega] \in H^2_\text{dR}(\Xf)$ of a differential $2$-form on~$\Xf$,
\begin{equation}
  [C] \cdot \Theta(\Omega) = \int_C \Omega.
\end{equation}
Moreover, if~$\Omega$ is a holomorphic $2$-form, then~$\int_C \Omega = 0$ because the restriction of~$\Omega$ to the complex $1$-dimensional
subvariety~$C$ vanishes.
In particular
$[C] \cdot \omega_f = 0$.
It turns out that this condition characterizes the elements of $\Pic(\Xf)$.

More precisely, let~$H^{1,1}(\Xf) \subset H_\mathbb{C}$ denote the space orthogonal to $\omega_f$ and $\overline{\omega}_f$, the conjugate of $\omega_f$, with respect to the intersection form.
This space is a direct summand in the Hodge decomposition of~$H^2(\Xf, \mathbb{C})$.

The Lefschetz (1,1)-theorem \parencite[163]{GriffithsHarris_1978} asserts that the lattice of integer relations coincide with the Picard group:
\begin{equation}\label{eq:7}
  \Pic(\Xf) = H_\mathbb{Z} \cap \Hii(\Xf).
\end{equation}
Noting that for any~$\gamma \in H_\mathbb{Z}$, $\overline \gamma = \gamma$, where~$\overline \gamma$ denotes the complex conjugate,
we have $\overline{\omega_f \cdot \gamma} = \overline{\omega}_f \cdot \gamma$,
so that~\eqref{eq:7} becomes
\begin{equation}\label{eq:8}
  \Pic(\Xf) = \left\{ \gamma \in H_\mathbb{Z} \st \gamma \cdot \omega_f = 0 \right\}.
\end{equation}

\subsection{A deformation argument}
\label{sec:deformation-argument}

Let~$\gamma_1,\dotsc,\gamma_{22}$ be a basis of~$H_\mathbb{Z}$.
The space $H_\mathbb{R}$ (resp.~$H_\mathbb{C}$) is endowed with the coefficient wise Euclidean (resp. Hermitian) norm
\begin{equation}\label{eq:16}
  \bigg\| \sum_{i=1}^{22} x_i \gamma_i \bigg\|^2 \eqdef \sum_{i=1}^{22} \abs{x_i}^2.
\end{equation}

For~$\gamma \in H_\mathbb{Z}$, if~$\abs{\gamma \cdot \omega_f}$ is small enough,
then~$\gamma$ is close to being an integer relation between the periods of~$\Xf$.
We want to argue that, in this case, $\gamma$ is a genuine integer relation between the periods of~$X_g$ for some polynomial~$g \in U_4$ close to~$f$.

Recall $f,g\in \widetilde{U}_4$ means $f$ and $g$ are smooth quartics with second cohomology identified with $H_\Z$. The space $\widetilde{U}_4$ inherits a metric from $U_4$ so that $\widetilde{U}_4 \to U_4$ is locally isometric. The metric on $U_4 \subset R_4 \simeq \C^{35}$ is induced by an inner product. The choice of an inner product will change the distances but this is absorbed into the constants in the statements below.

Let~$f \in \widetilde{U}_4$ be fixed. For any $g \in R_4$ and $t \in \C$ small enough, the polynomials   
$f+tg \in R_4$ lift canonically to $\widetilde{U}_4$.  
For any~$\gamma \in H_\mathbb{C}$ we consider the map
\begin{equation}
\phi_{\gamma,g}(t) \eqdef \gamma \cdot \mathcal{P}(f+tg)\label{eq:14}
\end{equation}
which is well-defined and analytic in a neighbourhood of~$0$ in~$\mathbb{C}$.

\begin{lemma}\label{lem:technical1}
  There is a constant~$C > 0$, depending only on~$f$,
  such that for any~$\gamma \in H_\mathbb{C}$ satisfying~$\gamma\cdot h = 0$ 
  and~$\abs{\gamma \cdot \overline{\omega}_f} \|\omega_f\| \le \frac12 \|\gamma\| ( \omega_f \cdot \overline{\omega}_f )$,
  there is a monomial~$m\in R_4$ for which~$\abs{\phi_{\gamma, m}'(0)} \geq C \|\gamma\|$.
\end{lemma}

\begin{proof}
  Observe that for any monomial~$m\in R_4$, $\phi_{\gamma, m}'(0) = \gamma \cdot \ud_f \mathcal{P}(m)$, where~$\ud_f \mathcal{P}$ is the derivative at~$f$ of~$\mathcal{P}$.
  Let~$Q$ be the positive semidefinite Hermitian form defined on $H_\C$ by
  \begin{equation}
    \label{eq:2}
    Q(\gamma) \eqdef  \sum_m \abs{\gamma \cdot \ud_f \mathcal{P}(m)}^2,
  \end{equation}
  where the sum is taken over the monomials in~$m$.
  Since~$\max_m \abs{\phi_{\gamma, m}'(0)}^2 \geq \tfrac{1}{\dim R_4} Q(\gamma)$,
  it is enough to prove that~$Q(\gamma) \geq C \|\gamma\|$ for some constant~$C > 0$,
  when~$\gamma \cdot h = 0$ and~$\abs{\gamma \cdot \overline{\omega}_f} \|\omega_f\| \le \frac12 \|\gamma\| ( \omega_f \cdot \overline{\omega}_f )$.

  The form~$Q$ vanishes exactly on the orthogonal complement (for the intersection product)
  of  the tangent space $T_{\omega_f} \mathcal{D}$ of~$\mathcal{D}$ at~$\omega_f$. By~\eqref{eq:42},
  \begin{equation}
    T_{\omega_f} \mathcal{D} = \left\{ w\in H_{\mathbb{C}} \st w \cdot h = w \cdot \omega_f = 0 \right\}.
  \end{equation}
  So the kernel of~$Q$  is~$K \eqdef \C h + \C \omega_f$.
  Moreover, let~$E$ be the orthogonal complement of~$\C h + \C \overline{\omega}_f$ (still for the intersection product).
  Since~$h \cdot \omega_f = h \cdot \overline{\omega}_f = 0$, $h\cdot h = 4$ and~$\omega_f \cdot \overline{\omega}_f > 0$,
  we check that~$E \cap K = 0$.
  In particular, the form~$Q$ is positive definite on~$E$, so there is a constant~$C > 0$
  such that
  $Q(\eta) \geq C \|\eta\|$ for any~$\eta \in E$.
  This constant is easily computable as the smallest eigenvalue of the matrix of the restriction of~$Q$ on that space, in a unitary basis, for the Hermitian norm~$\|-\|$.

  Now, let~$\gamma$ such that~$\gamma\cdot h = 0$
  and
  \begin{equation}
    \abs{\gamma \cdot \overline{\omega}_f} \|\omega_f\| \le \frac12 \|\gamma\| ( \omega_f \cdot \overline{\omega}_f ).\label{eq:17}
  \end{equation}
  Let~$a \eqdef (\gamma \cdot \overline{\omega}_f) / (\omega_f \cdot \overline{\omega}_f)$, and $\eta \eqdef \gamma - a \omega_f$,
  so that~$\eta \cdot \overline{\omega}_f = 0$ and~$\eta \cdot h = 0$, that is~$\eta \in E$.
  Since~$\omega_f$ is in the kernel of~$Q$, we have~$Q(\eta) = Q(\gamma)$, and thus~$Q(\gamma) \geq C \|\eta\|$.
  Lastly, we compute that
  \begin{equation}
    \|\eta\| \geq \|\gamma\| - \abs{a} \|\omega_f\|
             = \|\gamma\| - \abs{\frac{\gamma \cdot \overline{\omega_f}}{\omega_f \cdot \overline{\omega}_f}} \|\omega_f\|
             \geq \frac12 \|\gamma\|,
  \end{equation}
  using \eqref{eq:17}. So~$Q(\gamma) \geq \frac12 C\|\gamma\|$.
\end{proof}

The next statement is proved using the following result of \textcite{Smale_1986}.
Let~$\phi$ be an analytic function on a maximal open disc around~$0$ in~$\mathbb{C}$ with~$\phi'(0) \neq 0$.
We define
\begin{equation}
  \gsmale(\phi) \eqdef \sup_{k\geq 2} \abs{\frac{1}{k!} \frac{ \phi^{(k)}(0)}{\phi'(0)}}^{\frac{1}{k-1}} \text{ and } \bsmale(\phi) \eqdef \abs{\frac{\phi(0)}{\phi'(0)}}.
\end{equation}
If~$\bsmale(\phi) \gsmale(\phi) \leq \frac{1}{34}$, then there is a~$t \in \mathbb{C}$ such that~$\abs{t} \leq 2 \bsmale(\phi)$
and~$\phi(t) = 0$
\parencites{Smale_1986}[see also][Chapter~8, Theorem~2]{BlumCuckerShubSmale_1998}.

\begin{proposition}\label{prop:perturbation}
  For any~$f\in \widetilde{U}_4$, there exists~$\Cf$ and~$\epsf > 0$
  such that for all~$\epsilon < \epsf$ the following holds.
  For any~$\gamma \in H_\mathbb{R}$, if~$\gamma\cdot h = 0$ and~$\abs{\gamma\cdot \omega_f} \leq \epsilon \|\gamma\|$
  then there is a monomial~$m \in R_4$ and~$t\in \mathbb{C}$ such that
  $\abs{t} \leq C_f\epsilon$ and~$ \gamma \cdot \omega_{f+tm} = 0$.
\end{proposition}

\begin{proof}
  Let~$\gamma \in H_\mathbb{R}$ such that $\gamma \cdot h = 0$ and
  \begin{equation}
    \abs{\gamma\cdot \omega_f} \leq \left(\frac{\omega_f\cdot \overline{\omega}_f}{2 \|\omega_f\|}\right) \|\gamma\|.
  \end{equation}
  Since~$\gamma$ has real coefficients, we have~$\abs{\gamma \cdot \omega_f} = \abs{\gamma \cdot \overline{\omega}_f}$
  and we may apply Lemma~\ref{lem:technical1}
  to obtain a monomial~$m$ and a constant $C$ such that
  \begin{equation}
    \abs{\phi_{\gamma, m}'(0)} \geq C \|\gamma\|.
  \end{equation}
  It follows in particular that
  \begin{equation}\label{eq:40}
    \bsmale(\phi_{\gamma,m}) \leq \frac{\abs{\gamma \cdot \omega_f}}{C \|\gamma\|}.
  \end{equation}
  Moreover, for any~$k \geq 2$, and using~$C \leq 1$,
  \begin{align}
    \label{eq:39}
    \abs{\frac{1}{k!} \frac{ \phi_{\gamma,m}^{(k)}(0)}{\phi_{\gamma,m}'(0)}}^{\frac{1}{k-1}}
    &\leq C^{-1}  \abs{\frac{ \phi_{\gamma,m}^{(k)}(0)}{k! \|\gamma\|}}^{\frac{1}{k-1}}
    = C^{-1} \abs{ \frac{\gamma}{\|\gamma\|} \cdot \frac{1}{k!}\ud^k_f \mathcal{P}(m, \dotsc, m) }^{\frac{1}{k-1}} \\
    &\leq C^{-1} \tn{ \tfrac{1}{k!} \ud_f^k \mathcal{P} }^{\frac{1}{k-1}},
  \end{align}
  where~$\ud_f^k \mathcal{P} : R_4^k \to H_\C$ is the $k$th higher derivative of~$\mathcal{P}$ at~$f$ and
  where $\tn{\cdot}$ is the operator norm defined as
  \begin{equation}
    \tn{ \tfrac{1}{k!} \ud_f^k \mathcal{P} } \eqdef \sup_{\gamma\in H_{\mathbb{C}}} \sup_{h_1,\dotsc,h_k} \frac{ \abs{ \gamma \cdot \tfrac{1}{k!} \ud^k_f \mathcal{P}(h_1,\dotsc,h_n)} }{\|\gamma\| \|h_1\| \dotsb \|h_n\|},
  \end{equation}
  with supremum taken over $h_1,\dotsc,h_n \in \mathbb{C}[w,x,y,z]_4$.
  It follows that
  \begin{equation}\label{eq:41}
    \gsmale(\phi_{\gamma, m}) \leq C^{-1} \sup_{k \geq 2}  \tn{\tfrac{1}{k!}  \ud_f^k \mathcal{P} }^{\frac{1}{k-1}}.
  \end{equation}
  Let~$\Gamma$ denote the supremum on the right-hand side of~\eqref{eq:41}.
  By Smale's theorem, together with~\eqref{eq:40} and~\eqref{eq:41},
  if $\abs{\gamma \cdot \omega_f} \leq \frac{1}{34} C^2 \Gamma^{-1} \|\gamma\|$,
  then there is a~$t \in \mathbb{C}$ such that~$\abs{t} \leq 2 C^{-1} \abs{\gamma \cdot \omega_f} \|\gamma\|^{-1}$
  and~$\gamma \cdot \mathcal{P}(f + tm) = 0$.
  The claim follows with~$\Cf \eqdef 2 C^{-1}$ and
  \begin{equation}
  \epsf \eqdef \min\left(\tfrac{1}{34} C^2 \Gamma^{-1}, \frac{\omega_f\cdot \overline{\omega}_f}{2 \|\omega_f\|}\right). \qedhere\qed
  \end{equation}
\end{proof}

The constants~$\Cf$ and~$\epsf$ are actually computable with simple algorithms.
The constant from Lemma~\ref{lem:technical1} is not hard to get with elementary linear algebra.
It only remains to compute an upper bound for~$\Gamma$.
We address this issue in \S\ref{sec:inverse_function}.

\begin{corollary}\label{coro:picard-perturbation}
  For any~$f\in \widetilde{U}_4$,
  any~$\epsilon < \epsilon_f$,
  and any~$\gamma \in H_\mathbb{Z}$,
  if~$\abs{\gamma \cdot \omega_f}\leq \frac14 \epsilon$
  then there exists a monomial~$m\in R_4$ and~$t\in \mathbb{C}$
  such that~$\abs{t} \leq C_f \epsilon$ and~$\gamma \in \Pic(X_{f+tm})$.
\end{corollary}

\begin{proof}
  We may assume that~$\gamma \cdot \omega_f \neq 0$ (otherwise choose any~$m$ and~$t=0$).
  Let~$\gamma' = \gamma - \frac14 (\gamma\cdot h) h$.
  Since~$h\cdot h = 4$, we have~$\gamma'\cdot h = 0$.
  Moreover~$\gamma'\cdot \omega_f = \gamma\cdot \omega_f \neq 0$.
  In particular, $\gamma' \neq 0$ and since~$\gamma' \in \frac14 H_\mathbb{Z}$,
  we have~$\|\gamma'\| \geq \frac14$
  and then
  \begin{equation}
    \abs{\gamma'\cdot \omega_f} \leq 4 \|\gamma'\| \abs{\gamma\cdot \omega_f} \leq \epsilon \|\gamma'\|,
  \end{equation}
  and Proposition~\ref{prop:perturbation} applies.
\end{proof}

\subsection{Effective bounds for the higher derivatives of the period map}\label{sec:inverse_function}

In the proof of Proposition~\ref{prop:perturbation}, only the quantity~$\Gamma$
is not clearly computable. We show in this section how to compute an upper bound for $\Gamma$
using the Griffiths--Dwork reduction. We follow here \textcite{Griffiths_1969}.

Firstly, as a variant of~\eqref{eq:31} avoiding dehomogeneization, we write
\begin{equation}
  {\mathcal{P}}(f) = \left( \frac{1}{2\pi i} \int_{T(\gamma_i)} \frac{\volproj}{f} \right)_{1\leq i\leq 22}
\end{equation}
where~$\volproj$ is the projective volume form
\begin{equation}
  \volproj \eqdef w \ud x \ud y \ud z - x \ud w \ud y \ud z + y \ud w \ud x \ud z - z\ud w \ud x \ud y.
\end{equation}
For any~$k > 0$ and~$a \in R_{4k-4}$,
we denote
\begin{equation}
  \int \frac{a \volproj}{f^k} \eqdef \left( \frac{1}{2\pi i} \int_{T(\gamma_i)} \frac{a \volproj}{f^k} \right)_{1\leq i\leq 22} \in H_\mathbb{C}.
\end{equation}
For any~$h \in R_4$ close enough to~0, we have the power series expansion
\begin{equation}\label{eq:34}
  \int \frac{\volproj}{f+h} = \sum_{k\geq 1} (-1)^{k-1} \int \frac{h^{k-1} \volproj}{f^k}.
\end{equation}

\begin{proposition}\label{prop:bound-higher-derivatives-period-map}
  For any~$k \geq 3$, there is a linear map~$G_k \colon R_{4k-4} \to R_8$ such that
  \[ \int \frac{a}{f^k} \volproj = \int \frac{G_k(a)}{f^3} \volproj. \]
  Moreover, there is a computable constant~$C$, which depends only on~$f$, such that for any~$k\geq 3$, $\tn{G_k} \leq C^{k-3}$,
  where~$R$ is endowed with the $1$-norm \eqref{eq:1norm}.
\end{proposition}

\noindent Before we begin the proof of proposition, let us show that this is enough to bound~$\Gamma$.
Let~$A \colon a \in R_8 \mapsto \int \frac{a}{f^3} \volproj \in H_\mathbb{C}$,
then, using~\eqref{eq:34} we obtain
\begin{equation}
  \int \frac{\volproj}{f+h} = \sum_{k \geq 1} (-1)^{k-1} A(G_k(h^{k-1})),
\end{equation}
and it follows that
\begin{equation}
  \tfrac{1}{k!} \ud_f^k {\mathcal{P}}(h_1,\dotsc,h_k) = (-1)^k A( G_{k+1}(h_1\dotsb h_k)).
\end{equation}
In particular,
\begin{align}
  \left\|\tfrac{1}{k!} \ud_f^k {\mathcal{P}}(h_1,\dotsc,h_k)\right\|
 & \leq  \tn{A} \tn{G_{k+1}} \|h_1 \dotsb h_n\|_1 \\
  &\leq \tn{A} \tn{G_{k+1}} \|h_1\|_1 \dotsb \|h_n\|_1,
\end{align}
and therefore
$  \tn{ \tfrac{1}{k!} \ud_f^k {\mathcal{P}}} \leq \tn{A} C^{k+1}$,
from which we get 
\begin{equation}
  \Gamma \leq C \max \left(\tn{A} C^2, 1\right)
\end{equation}

Let us remark on how to bound the operator norm of $A$ in practice. The period integrals can be approximated to arbitrary precision and with rigorous error bounds as in \parencite{Sertoz_2019}. This construction gives a small neighbourhood of $A$ in the matrix space. In practice, we represent this neighbourhood as a matrix $A'$ of complex balls and compute the operator norm of $A'$ as usual but using complex ball arithmetic. This will return a real open interval containing $\tn{A} \neq 0$. If the precision is high enough, $0$ will not be contained in the closure of this interval and we can take the lower bound of the interval.

\subsubsection{Proof of Proposition~\ref{prop:bound-higher-derivatives-period-map}}

Let~$R = \mathbb{C}[w,x,y,z]$.
We define two families of maps for this proof.
First, for~$d \geq 12$, a multivariate division map $Q_d \colon R_d \to R_{d-3}^4$,
such that for any~$a \in R_d$,
\begin{equation}\label{eq:3}
  a = \sum_{i=0}^3 Q_d(a)_i \, \partial_i f.
\end{equation}
Note that such a map exists as soon as~$d \geq 12$ by a theorem due to Macaulay \parencite[see][Corollaire, p.~169]{Lazard_1977}.
The choice of~$Q_d$ is not unique. We fix $Q_{12}$ arbitrarily
and define $Q_d(a)$, for~$d > 12$ and~$a\in R_{12}$, as follows.
Write~$a = \sum_{i=0}^3 x_i a_i$,
in such a way that the terms of the sum have disjoint monomial support, and define
\begin{equation}
  Q_d(a) = \sum_{i=0}^3 x_i Q_{d-1}(a_i).
\end{equation}
It is easy to check that this definition satisfies~\eqref{eq:3}.

Second, for~$k \geq 3$, we define~$G_k \colon R_{4k-4} \to R_{8}$
as follows. Begin with $G_3 = \mop{id}$ and then define $G_k$ for~$k \geq 4$ inductively as follows.
For~$a \in R_{4k-4}$ we write~$(b_0,\dotsc, b_3) = Q_{4k-4}(a)$ and define
\begin{equation}\label{eq:37}
  G_k (a) \eqdef G_{k-1} \left( \frac{1}{k-1}(\partial_0 b_0 + \dotsb + \partial_3 b_3) \right).
\end{equation}
This map is the Griffiths--Dwork reduction, and it satisfies
\begin{equation}
  \int_\gamma \frac{a \Omega}{f^k} = \int_\gamma \frac{G_k(a) \Omega}{f^{3}}.
\end{equation}

\begin{lemma}\label{lem:tnQ}
  For any~$d \geq 12$,
  $\tn{Q_d} \leq \tn{Q_{12}}$,
  where~$R$ is endowed with the $1$-norm and~$R^4$ with the norm~$\|(f_0,\dotsc,f_3)\|_1  \eqdef \|f_0\|_1+\dotsb+\|f_3\|_1$.
\end{lemma}

\begin{proof}
  For any~$a\in R_d$,
  \begin{align}
    \|Q_d(a)\|_1 &= \sum_{i=0}^3 \|Q_d(a)_i\|_1 \leq \sum_{i=0}^3 \sum_{j=0}^3 \|x_j Q_{d-1}(a_j)_i \|_1\\
               &= \sum_{i=0}^3 \sum_{j=0}^3 \|Q_{d-1}(a_j)_i \|_1 = \sum_j \|Q_{d-1}(a_j)\|_1 \\
               &\leq \tn{Q_{d-1}} \sum_j \|a_j\|_1 = \tn{Q_{d-1}} \|a\|_1,
  \end{align}
  using, for the last equality, that the terms~$a_j$ have disjoint monomial support.
\end{proof}

\begin{lemma}\label{lem:tn-Gk}
  For any~$k \geq 3$,
  $\tn{G_k} \leq ( 4\tn{Q_{12}} )^{k-3}$,
  where~$R$ is endowed with the $1$-norm.
\end{lemma}

\begin{proof}
  We proceed by induction on~$k$ (the base case $k=3$ is trivial since~$G_3=\mop{id}$).
  Let~$a \in R_{4k-4}$ and~$(b_0,\dotsc,b_3) = Q_{4k-4}(a)$.
  By~\eqref{eq:37}, we have
  \begin{align}
    \|G_k(a)\|_1 &\leq \frac{\tn{G_{k-1}}}{k-1} \left( \| \partial_0 b_0 \|_1 + \dotsb + \|\partial_3 b_3\|_1 \right).
  \end{align}
  By induction hypothesis, $\tn{G_{k-1}} \leq (4 \tn{Q_{12}} )^{k-4}$
  and moreover $\|\partial_i b_i\|_1 \leq (4k-7) \|b_i\|_1$, since each~$b_i$ has degree~$4k-7$.
  If follows that
  \begin{align}
    \|G_k(a)\|_1 &\leq (4\tn{Q_{12}})^{k-4} \frac{4k-7}{k-1} \left( \|b_0\|_1 + \dotsb + \|b_3\|_1 \right).
  \end{align}
  Next, we note that~$\|b_0\|_1 + \dotsb + \|b_3\|_1 = \|Q_{4k-4}(a)\|_1$
  and, by Lemma~\ref{lem:tnQ}, $\tn{Q_{4k-4}(a)}\leq \tn{Q_{12}}$.
  Therefore
  \begin{equation}
    \|G_k(a)\|_1 \leq (4 \tn{Q_{12}})^{k-3} \|a\|_1,
  \end{equation}
  and the claim follows.
\end{proof}

\section{The Noether--Lefschetz locus}\label{sec:NL}

\subsection{Basic properties}
We define the Noether--Lefschetz locus for quartic surfaces and review a few classical properties, especially algebraicity, with a view towards Theorem~\ref{thm:nl-bounds} about the degree and the height of the equations defining the components of the Noether--Lefschetz locus.

\subsubsection{Definition}\label{sec:defNL}

The Noether--Lefschetz locus of quartics~$\NL$ is the set of all~$f\in U_4$
such that the rank of~$\Pic(\Xf)$ is at least~2. Equivalently, in view of~\eqref{eq:8},
$\NL$ is the set of quartic polynomials~$f$ whose primitive periods~\eqref{eq:intro} are $\Z$-linearly dependent. %

The set $\NL$ is locally the union of smooth analytic hypersurfaces in~$U_4$. To see this,
let~$\wNL$ be the lift of~$\NL$ in the universal covering~$\widetilde{U}_4$ of~$U_4$. 
Recall $\mathcal{P}\colon \widetilde{U}_4 \to \mathcal{D}$ is the period map. 
The Lefschetz (1,1)-theorem implies 
\begin{equation}\label{eq:1}
  \wNL = \bigcup_{\gamma \in H_\mathbb{Z} \setminus \mathbb{Z} h} \mathcal{P}^{-1} \left\{ w \in \mathcal{D} \st w \cdot \gamma = 0 \right\}.
\end{equation}
That is, $\wNL$ is the pullback of smooth hyperplane sections of $\mathcal{D}$. 
Since~$\mathcal{P}$ is a submersion, $\wNL$ is the union of smooth analytic hypersurfaces. It follows that $\NL$ is locally the union of smooth analytic hypersurfaces.

We break $\NL$ into algebraic pieces as follows. For any integers~$d$ and~$g$,
let~$\NL_{d,g}$ be the set
\begin{equation}\label{eq:4}
  \NL_{d,g} = \left\{ f \in U_4 \st \exists \gamma \in \Pic(\Xf) \setminus \mathbb{Z} h: \gamma\cdot h = d \text{ and } \gamma\cdot \gamma=2g - 2 \right\},
\end{equation}
By replacing~$\gamma$ by~$\gamma+h$ or~$-\gamma$, we observe that
\begin{equation}\label{eq:10}
  \NL_{d,g} = \NL_{d+4, g+d+2} = \NL_{-d, g}.
\end{equation}
In particular, $\NL_{d,g}$ is equal to some~$\NL_{d',g'}$ with~$d' > 0$ and~$g' \geq 0$, so that
\begin{equation}
  \NL = \bigcup_{d > 0} \bigcup_{g \geq 0} \NL_{d,g}.
\end{equation}

For~$\gamma \in H_\mathbb{Z}$, let~$\Delta(\gamma) = (h \cdot \gamma)^2 - 4 \gamma \cdot \gamma$.
It is the negative of the discriminant of the lattice generated by~$h$ and~$\gamma$ in~$H_\mathbb{Z}$,
with respect to the intersection product (and it is zero if~$\gamma \in \mathbb{Z} h$).
It follows from the Hodge index theorem \parencite[see][Theorem~V.1.9]{Hartshorne_1977}
that for any~$f\in U_4$ and any~$\gamma \in \Pic(\Xf)$, $\Delta(\gamma) \geq 0$, with equality if and only if~$\gamma \in \mathbb{Z} h$.
If~$\gamma\cdot h = d$ and~$\gamma\cdot \gamma = 2g-2$, then~$\Delta(\gamma) = d^2 - 8g+8$.
We obtain therefore that for any~$d > 0$ and~$g \geq 0$,
\begin{equation}
  \NL_{d,g} = \footnotesize
  \begin{cases}
    \left\{ f \in U_4 \st \exists \gamma \in \Pic(\Xf) : \gamma\cdot h = d \text{ and } \gamma\cdot \gamma=2g - 2 \right\} & \text{if $d^2 > 8g-8$}\\
    \varnothing &\text{otherwise.}
  \end{cases}
\end{equation}

It is in fact more natural to introduce, for~$\Delta > 0$, the following locus
\begin{align}\label{eq:23}
  \NL_\Delta &\eqdef \left\{ f \in U_4 \st \exists \gamma \in \Pic(\Xf) : \Delta(\gamma) = \Delta \right\}\\
             &= \bigcup_{\substack{d > 0\\ d^2 \equiv \Delta \text{ mod } 8}} \NL_{d, \frac{d^2-\Delta}{8} + 1}.
\end{align}
Due to~\eqref{eq:10}, $\NL_\Delta$ reduces to a single~$\NL_{d,g}$. Namely,
\begin{equation}\label{eq:11}
  \NL_\Delta =
  \begin{cases}
    \NL_{4t, 2t^2 + \frac{8-\Delta}8},&\text{if $\Delta\equiv 0 \ \mathrm{mod}\ 8$,}\\
    \NL_{4t+1, 2t^2 + t + \frac{9-\Delta}8} & \text{if $\Delta\equiv 1 \ \mathrm{mod}\ 8$,} \\
    \NL_{4t+2, 2t^2 + 2t + \frac{12 - \Delta}8}, &\text{if $\Delta\equiv 4 \ \mathrm{mod}\ 8$,} \\
    \varnothing & \text{otherwise,}
  \end{cases}
\end{equation}
where~$t = \lceil \frac14\sqrt{\Delta} \rceil$. Conversely, each $\NL_{d,g} = \NL_{d^2 - 8g +8}$.

\subsubsection{Algebraicity}
For any~$d > 0$ and~$g\geq 0$, the set~$\NL_{d,g}$ is either empty or an
algebraic hypersurface in~$U_4$.
This is a classical result \parencite[e.g.][Theorem~3.32]{Voisin_2003}
which we recall here to obtain an explicit algebraic description of~$\NL_{d,g}$.

\begin{lemma}\label{lem:curve-pic}
  For any~$f\in U_4$, $d > 0$ and~$g\geq 0$ we have: $f \in \NL_{d,g}$ if and only if~$\Xf$ contains an effective divisor
  with Hilbert polynomial~$t\mapsto dt + 1-g$.
\end{lemma}

\begin{proof}
  Assume that~$\Xf$ contains an effective divisor~$C$ with Hilbert polynomial~$t\mapsto td + 1-g$.
  Since $\Xf$ is smooth, $C$ is a locally principal divisor and gives an element~$\gamma$ of~$\Pic \Xf$.
  The integer~$d$ is the degree of $C$, so it is the number of points in the intersection with a generic hyperplane, that is~$d = \gamma\cdot h$. 
  Moreover, $g$ is the arithmetic genus of~$C$, which is determined by $2g-2 = \gamma\cdot \gamma$ \parencite[Ex.~III.5.3(b) and~V.1.3(a)]{Hartshorne_1977}.
  So~$f \in \NL_{d,g}$. 

  Conversely,
  let~$f \in \NL_{d,g}$. By definition, there is a divisor~$C$ on~$\Xf$ such that its class~$\gamma$ in~$\Pic \Xf$
  satisfies~$\gamma \cdot h=d$ and~$\gamma\cdot \gamma = 2g-2$.
  From the Riemann--Roch theorem for surfaces \parencite[V.1.6]{Hartshorne_1977} we get:
  \begin{equation*}
    \dim H^0(X, \mathcal{O}_X(C)) + \dim H^0(X, \mathcal{O}_X(-C)) \geq \tfrac12  \gamma\cdot \gamma + 2 = g+1 > 0
  \end{equation*}
  so that either $C$ or~$-C$ must be linearly equivalent to an effective
  divisor. Since~$\gamma \cdot h > 0$, $-C$ can not be effective and therefore $C$ must
  be. As above, the Hilbert polynomial of~$C$ is given by~$t\mapsto dt+1-g$.
\end{proof}

In light of Lemma~\ref{lem:curve-pic}, the algebraicity of~$\NL_{d,g}$ is proved
by using the Hilbert scheme~$\mathcal{H}_{d,g}$. The Hilbert scheme~$\mathcal{H}_{d,g}$ of
degree~$d$ and genus~$g$ curves in~$\mathbb{P}^3$ is a projective scheme that
parametrizes all the subschemes of~$\mathbb{P}^3$ whose Hilbert polynomial is
$t\mapsto dt + 1 - g$.

The Hilbert scheme~$\mathcal{H}_{d,g}$ may contain components that are not desirable for our purposes. 
For example~$\mathcal{H}_{3,0}$, which contains twisted cubics in~$\mathbb{P}^3$,
contains two irreducible components \parencite{PieneSchlessinger_1985}: a $12$-dimensional component
that is the closure of the space of all smooth cubic rational curves in~$\ppp$; and a $15$-dimensional component
parametrizing the union of a plane cubic curve with a point in~$\ppp$. 
We would be only interested in the first, not in the second component. 
So we introduce~$\mathcal{H}'_{d,g}$, the union of components of~$\mathcal{H}_{d,g}$ obtained by
removing the components that does not correspond to locally-complete-intersection pure-dimensional subschemes of $\ppp$.

When~$d^2 > 8g-8$, Lemma~\ref{lem:curve-pic} can be rephrased as
\begin{equation}\label{eq:9}
  \NL_{d,g} = \mop{proj}_1 \left\{ (f, C) \in U_4 \times \mathcal{H}_{d,g}' \st C \subset \Xf \right\},
\end{equation}
where~$\mop{proj}_1$ denotes the projection~$U_4 \times \mathcal{H}_{d,g}' \to U_4$.
Since $\mathcal{H}_{d,g}'$ is a projective variety, and the condition~$C\subset \Xf$ is algebraic,
this shows that~$\NL_{d,g}$ is a closed subvariety of~$U_4$ (for more details about this construction, see \cite[\S 3.3]{Voisin_2003}). 

We note furthermore that $\NL_{d,g}$ is clearly invariant under the action of the Galois group of algebraic numbers. 
Therefore, it can be defined over the rational numbers.

As a consequence, for any nonnegative integers~$d$ and~$g$,
there is a squarefree primitive homogeneous polynomial~$\pNL_{d,g} \in \mathbb{Z}[u_1,\dotsc,u_{35}]$ 
in the 35~coefficients of the general quartic polynomial 
that is unique up to sign and
whose zero locus is~$\NL_{d,g}$ in~$U_4$.
Similarly, we define~$\pNL_\Delta$ upto sign as the unique squarefree primitive polynomial vanishing exactly on~$\NL_\Delta$.

\subsection{Height of multiprojective varieties}

The mainstay of our results is a bound on the degree and size of the coefficients of the polynomials~$\pNL_{d,g}$.
The determination of these bounds is based on~\eqref{eq:9} and
involves the theory of heights of multiprojective varieties
as developped by \textcite{DAndreaKrickSombra_2013}, and, before them,
\textcite{BostGilletSoule_1991,Philippon_1995,KrickPardoSombra_2001,Remond_2001,Remond_2001a}, among others.
We recall here the results that we need, following \textcite{DAndreaKrickSombra_2013}.

\subsubsection{Heights of polynomials}

Let $f = \sum_{\alpha} c_\alpha {\mathbf x}^\alpha \in \mathbb{C}[x_1,\dots,x_n]$. We recall the following different measures of height of $f$:
\begin{align}
  \norm{f}_1 &\eqdef  \sum_\alpha \abs{c_\alpha}, \label{eq:1norm} \\
  \norm{f}_{\sup} &\eqdef  \sup_{\abs{x_1}=\dotsb=\abs{x_n}=1} \abs{f(\mathbf x)},\\
  m(f) &\eqdef \int_{[0,1]^n} \log\abs{f \left( e^{2\pi i t_1}, \dotsc, e^{2\pi i t_n} \right) } \ud t_1 \dotsb \ud t_n.
\end{align}

\begin{lemma}[{\cite[Lemma~2.30]{DAndreaKrickSombra_2013}}]
  \label{lem:sup_dominates}
  For any homogeneous polynomial $f \in \mathbb{C}[x_1,\dots,x_n]$,
  \begin{equation*}
    \exp(m(f)) \le \norm{f}_{\sup} \le \norm{f}_{1} \leq \exp(m(f)) (n+1)^{\deg f}.
  \end{equation*}
\end{lemma}

\subsubsection{The extended Chow ring}

The extended Chow ring
\parencite[Definition~2.50]{DAndreaKrickSombra_2013} is a tool to track a measure of height of multiprojective
varieties when performing intersections and projections.
We present here a very brief summary.
Bold letters refer
to multi-indices and all varieties are considered over~$\mathbb{Q}$.
Let~$\mathbf n \in \mathbb{N}^{r}$ and
let~$\mathbb{P}^{\mathbf{n}}$ be the multiprojective
space~$\mathbb{P}^{\mathbf n} = \mathbb{P}^{n_1}\times \dotsb \times \mathbb{P}^{n_m}$.

An \emph{algebraic cycle} is a finite $\mathbb{Z}$-linear combination $\sum_{V} n_V V$ of irreducible subvarieties of~$\mathbb{P}^\bfn$.
The \emph{irreducible components} of an algebraic cycle as above are the irreducible varieties~$V$ such that~$n_V \neq 0$.
An algebraic cycle is \emph{equidimensional} if all its irreducible components have the same dimension.
An algebraic cycle is \emph{effective} if~$n_V \geq 0$ for all~$V$.
The \emph{support} of~$X$, denoted by~$\mop{supp} X$, is the union of the irreducible components of~$X$.

Let~$A^*(\mathbb{P}^{\mathbf{n}}; \mathbb{Z})$ be the extended Chow ring, namely
\begin{equation}\label{eq:An}
  A^*(\mathbb{P}^{\mathbf{n}}; \mathbb{Z}) \eqdef \mathbb{R}[\eta, \theta_1,\dotsc,\theta_m]/(\eta^2, \theta_1^{n_1+1},\dotsc,\theta_m^{n_m+1}),
\end{equation}
where $\theta_i$ is the class of the pullback of a hyperplane from $\mathbb{P}^{n_i}$ and $\eta$ is used to keep track of heights of varieties.   
For two elements~$a$ and~$b$ of this ring, we write~$a \leq b$ when the coefficients of~$b-a$ in the monomial basis are nonnegative.

To an algebraic cycle~$X$ of~$\mathbb{P}^{\mathbf{n}}$ we associate 
an element~$[X]_\mathbb{Z}$ of~$A^*(\mathbb{P}^{\mathbf{n}}; \mathbb{Z})$
\parencite[Definition~2.50]{DAndreaKrickSombra_2013}.
If~$X$ is effective, then~$[X]_\mathbb{Z} \geq 0$.
The coefficients of the terms in $[X]_{\mathbb{Z}}$
for monomials not involving $\eta$ record the usual multi-degrees of $X$. The
terms involving $\eta$ record mixed canonical heights of $X$. The definition of
these heights is based on the heights of various Chow forms associated to $X$
\parencite[\S2.3]{DAndreaKrickSombra_2013}. For the computations in this paper,
we only need the following results.

Let~$f \in \mathbb{Z}[\mathbf{x_1},\dotsc,\mathbf{x_r}]$
be a nonzero multihomogeneous polynomial with respect to the group of variables~$\bx_1,\dotsc,\bx_n$.
We assume that~$f$ is \emph{primitive}, that is, the g.c.d.\ of the coefficients of $f$ is~1.
The element associated in $A^*(\mathbb{P}^{\mathbf{n}}; \mathbb{Z})$ to the hypersurface~$V(f) \subset \mathbb{P}^{\mathbf n}$
is \parencite[Proposition~2.53(2)]{DAndreaKrickSombra_2013}
\begin{equation}\label{eq:22}
  [V(f)]_{\mathbb{Z}} = m(f) \eta + \deg_{\mathbf{x}_1}(f) \theta_1 +  \dotsb + \deg_{\mathbf{x}_r}(f) \theta_r.
\end{equation}
To such a polynomial $f$, we also associate \parencite[Eq.~(2.57)]{DAndreaKrickSombra_2013}
\begin{equation}\label{eq:5}
  [f]_{\sup} \eqdef  \log(\|f\|_{\sup}) \eta + \deg_{\mathbf{x}_1}(f) \theta_1 +  \dotsb + \deg_{\mathbf{x}_r}(f) \theta_r.
\end{equation}

\subsubsection{Arithmetic Bézout theorem}

Let~$X$ be an effective cycle and~$H$ a hypersurface
in~$\mathbb{P}^\bfn$. They \emph{intersect properly} if no irreducible component
of~$X$ is in~$H$. When~$X$ and~$H$ intersect properly, ones defines an
intersection product~$X\cdot H$, that is an effective cycle supported on 
$X \cap H$. If~$X$ is
equidimensional of dimension~$r$, then~$X\cdot H$ is equidimensional of
dimension~$r-1$. 

The following statement is an arithmetic B\'ezout bound that not
only bounds the degree, as with the
classical B\'ezout bound, but also the height of an intersection.

\begin{theorem}[{\cite[Theorem~2.58]{DAndreaKrickSombra_2013}}]
  \label{thm:arithmetic-bezout}
  Let $X$ be an effective equidimensional cycle on $\mathbb{P}^{\mathbf n}$ and $f \in \mathbb{Z}[\mathbf x_1,\dotsc,\mathbf x_m]$.
  If~$X$ and~$V(f)$ intersect properly,
  then $[X \cdot V(f)]_{\Z} \le [X]_{\Z} \cdot [f]_{\sup}$.
\end{theorem}

This theorem can be applied (as in \cite[Corollary~2.61]{DAndreaKrickSombra_2013}) to
bound the height of the irreducible components of a variety in terms of its
defining equations.

\begin{proposition}\label{prop:bound-height-from-defining-equations}
  Let~$Z \subset \MP$ be an equidimensional variety and let $X$ be $V(f_1,\dotsc,f_s) \cap Z$, where 
  $f_i$ is a multihomogeneous polynomial of multidegree at most~$\mathbf{d}$
  and sup-norm at most~$L$.
  Let~$X_r$ be the union of all the irreducible components of~$X$ of codimension~$r$ in~$Z$.
  Then
  \[ [X_r]_\mathbb{Z} \leq [Z]_\mathbb{Z} \left( \log(s L ) \eta + \sum_{i=1}^m d_i \theta_i \right)^r. \]
\end{proposition}

\begin{proof}
  Let~$(y_{ij})$ be a new group of variables, with~$1\leq i \leq r$ and~$1\leq j \leq s$.
  Let~$g_i \eqdef \sum_{j=1}^s y_{ij} f_j$
  and~$X' \eqdef V(g_1,\dotsc,g_r)$ in~$\mathbb{P}^{k} \times Z$, with~$k = rs-1$
  We first claim that~$\mathbb{P}^k \times X_r$ is a union of components of~$X'$.
  Indeed, let~$\xi_0$ be the generic point of~$\mathbb{P}^k$ and~$\xi_1$ be the generic point of a component~$Y$ of~$X_r$,
  so that~$\xi = (\xi_0,\xi_1)$ is the generic point of the component~$\mathbb{P}^k \times Y$ of~$\mathbb{P}^k \times X_r$.
  Since~$X$ has codimension~$r$ at~$\xi_1$, the generic linear combinations~$g_1,\dotsc,g_r$ form a regular sequence at~$\xi$ (in other words, they form a regular sequence at~$\xi_1$ for generic values of the~$v_{ij}$).
  Therefore, $X'$ has codimension~$r$ at~$\xi$. Since~$\mathbb{P}^k \times Y \subseteq X'$, it follows that~$\mathbb{P}^k \times Y$ is a component of~$X'$.

  Let~$X'_r$ be the union of the components of codimension~$r$ of~$X'$.
  The argument above shows that $[\mathbb{P}^k \times X_r]_\mathbb{Z} \leq [X'_r]_\mathbb{Z}$.
  Besides, by repeated application of \parencite[Corollary~2.61]{DAndreaKrickSombra_2013},
  \begin{equation}
    [X'_r]_\mathbb{Z} \leq [\mathbb{P}^k \times Z]_\mathbb{Z} \prod_{i=1}^r [g_i]_\text{sup}.
  \end{equation}
  We compute, using~\eqref{eq:22} that
  \begin{equation}
    [g_i]_\text{sup} \leq \log (sL) \eta + \theta_0 + \sum_{i=1}^s d_i \theta_i.
  \end{equation}
  Finally, we note that~$[\mathbb{P}^k \times X_r]_\mathbb{Z} = [X_r]_\mathbb{Z}$ and~$[\mathbb{P}^k \times Z]_\mathbb{Z} = [Z]_\mathbb{Z}$ \parencite[Proposition~2.51.3 and 2.66]{DAndreaKrickSombra_2013}.
\end{proof}

\begin{proposition}\label{prop:projection_formula}
  Let $X $ be an equidimensional closed subvariety of $\P^k \times \MP$ and  
  let~$Y \subset \MP$ be the projection of $X$. 
  If~$Y$ is equidimensional, then
  \begin{equation*}
    \theta_0^{k} [Y]_{\Z} \le \theta_0^{\dim X - \dim Y}[X]_{\Z} \quad \in \quad A^*(\mathbb{P}^k \times \mathbb{P}^{\mathbf{n}}; \mathbb{Z}),
  \end{equation*}
  where~$\theta_0$ is the variable attached to~$\mathbb{P}^k$ in the extended Chow ring of~$\mathbb{P}^k\times \MP$.
\end{proposition}
\begin{proof}
  We will argue by induction on $r \eqdef \dim X - \dim Y$. When $r=0$, this is
  \parencite[Proposition~2.64]{DAndreaKrickSombra_2013}.

  Suppose now that $r >0$ and $X$ is irreducible.
  Let~$\mathbb{Q}[\mathbf y, \mathbf x_1, \dotsc,\mathbf{x_m}]$ denote the multihomogeneous coordinate
  ring of~$\mathbb{P}^{k}\times\MP$.
  There is an~$i$, $0\leq i\leq k$,
  such that~$H \eqdef V(y_i) \subset \mathbb{P}^k\times \MP$
  intersects~$X$ properly (otherwise $X$ would be included in all~$V(y_i)$ and
  would be empty). Since the fibers of $X \to Y$ are positive dimensional, $H$
  intersects each fiber. In particular, the set-theoretical projections of $X$ and $X \cap H$ coincide. 
  As $X$ is irreducible, so is $Y$. In particular, there is an
  irreducible component $X' \subset X\cap H$ that maps to $Y$. By induction
  hypothesis applied to $X'$, $\theta_0^{k} [Y]_{\Z}\le \theta_0^{\dim X'-\dim Y}[X']_{\Z}$.
  Moreover, $[X']_{\mathbb{Z}} \leq [X]_{\mathbb{Z}} [y_i]_{\sup}$,
  and, in view of~\eqref{eq:5}, $[y_i]_{\sup} = \theta_0$.
  The claim follows.

  If $X$ is reducible, then we apply the inequality above to each of the
  irreducible components of $Y$ together with an irreducible component of $X$
  mapping onto that component.
\end{proof}

\subsection{Explicit equations for the Noether--Lefschetz loci}

Following \textcite{Gotzmann_1978}, \textcite{Bayer_1982}, and the exposition of \textcite{Lella_2012}, we
describe the equations defining the Hilbert schemes of curves in $\ppp$. 
An explicit description of the Noether--Lefschetz loci~$\NL_{d,g}$ follows.

\subsubsection{The Hilbert schemes of curves}

For~$d > 0$ and~$g \geq 0$ let~$\ch_{d,g}$ be the Hilbert scheme of curves of
degree~$d$ and genus~$g$ in~$\mathbb{P}^3$. It parametrizes subschemes
of~$\mathbb{P}^3$ with Hilbert polynomial~$p(m) \eqdef dm + 1-g$.
Smooth curves in~$\ppp$ of degree~$d$ and genus~$g$, in particular,
have Hilbert polynomial~$p(m)$.
Let~$R =\C[w,x,y,z]$ be the homogeneous coordinate ring of~$\ppp$.
For~$m\geq 0$, let~$R_m$ denote the $m$th homogeneous part of~$R$
and let $q(m) = \dim R_m - p(m)$.

The Hilbert scheme~$\mathcal{H}_{d,g}$ can be realized in a Grassmannian variety
as follows.
A subscheme~$X$ of~$\ppp$
is uniquely defined by a saturated homogeneous ideal~$I$ of~$R$.
If the Hilbert polynomial of~$X$ is~$p$, then~$I$ is the saturation of the ideal generated by the degree~$r$ slice~$I_r \eqdef I \cap R_r$
\parencites[]{Gotzmann_1978}[\S II.10]{Bayer_1982},
where
\begin{equation}\label{eq:Gotzmann}
  r = \binom{d}{2} + 1 - g,
\end{equation}
is the Gotzmann number of~$p$ \parencite[\S II.1.17]{Bayer_1982}.
For practical reasons, we need~$r \geq 4$, so we define instead
\begin{equation}
  r = \max\left(\binom{d}{2} + 1 - g, 4\right).
\end{equation}
So~$X$ is
entirely determined by~$I_r$, which is a $q(r)$-dimensional subspace of~$R_r$.

Let~$\mathbb{G}$ be the Grassmannian variety of $q(r)$-dimensional subspaces
of~$R_r$. As a set, one can construct~$\mathcal{H}_{d,g}$ as the subset of
all~$\Xi \in \mathbb{G}$ such that the ideal generated by~$\Xi$ in~$R$ defines a
subscheme of~$\ppp$ with Hilbert polynomial~$p$.
In fact, $\mathcal{H}_{d,g}$ is a subvariety that is defined by the following condition \parencite[\S VI.1]{Bayer_1982}:
\begin{equation}\label{eq:6}
  \mathcal{H}_{d,g} = \left\{ \Xi \in \mathbb{G} \st \dim( R_1 \Xi ) \leq q(r+1) \right\},
\end{equation}
where~$R_1$ is the space of linear forms in~$w,x,y,z$, so that~$R_1 \Xi$ is a subspace of~$R_{r+1}$.

Several authors gave explicit equations for~$\mathcal{H}_{d,g}$ in the
Plücker coordinates \parencite{Bayer_1982,Grothendieck_1961b,Gotzmann_1978,BrachatLellaMourrainRoggero_2016}.
We will prefer here a more direct path that avoids the Plücker embedding.

\subsection{Equations for the relative Hilbert scheme}\label{sec:equat-relat-hilb}

Define the relative Hilbert scheme of curves inside quartic surfaces
\begin{equation}\label{eq:rel_hilb_scheme}
  \ch_{d,g}(4) \eqdef \{(f,C) \in \mathbb{P}(R_4) \times \ch_{d,g} \mid C \subset V(f) \},
\end{equation}
for each $d > 0$, $g\ge 0$.

We define the following auxiliary spaces to better describe~\eqref{eq:rel_hilb_scheme}. First, define the following ambient space
\begin{equation}
  \mathcal{A} \eqdef \mathbb{P}(R_4) \times \mathbb{P}\left(\mop{End}(\mathbb{C}^{q(r) - N_{r-4}}, R_r)\right) \times \mathbb{P}\left( \mop{End}(R_{r+1}, \mathbb{C}^{p(r+1)}) \right).
\end{equation}
Second, let $\cb = \{(f,\phi,\psi) \in \ca\}$ be the set of all triples satisfying the conditions
\begin{enumerate}[(i)]
  \item\label{item:1} $R_{r-3} f \subseteq \ker \psi$,
  \item\label{item:1bis}  $R_1 \mop{im}(\phi) \subseteq \ker \psi$,
  \item\label{item:2} $\mop{im} \phi \cap R_{r-4}f = 0$,
  \item\label{item:3} $\phi$ and~$\psi$ are full rank.
\end{enumerate}
Finally, we denote by $\overline{\cb}$ the Zariski closure of $\cb$.

\begin{lemma}
  The map~$\mathcal{B} \to \mathcal{H}_{d,g}(4)$ defined by~$(f, \phi, \psi) \mapsto (f, R_{r-4} f + \mop{im}\phi)$
  is well defined and surjective.
\end{lemma}

\begin{proof}
  Let~$(f, \phi, \psi) \in \mathcal{B}$ and let~$\Xi = R_{r-4} f + \mop{im} \phi$. 
  Constraint~\ref{item:3} implies that~$\mop{im}\phi$ has dimension~$q(r) - N_{r-4}$.
  Together with Constraint~\ref{item:2}, we have~$\dim \Xi = q(r)$.
  Moreover, Constraint~\ref{item:3} implies that~$\ker \psi$ has dimension~$q(r+1)$.
  In particular
  Since~$R_1 \Xi = R_{r-3} f + R_1 \mop{im}\phi$,
  Constraints~\ref{item:1} and~\ref{item:1bis} implies that~$R_1 \Xi$ has dimension at most~$q(r+1)$.
  So,~$\Xi \in \ch_{d,g}(4)$.
  Since~$R_{r-4} f \subseteq \Xi$,
  the polynomial~$f$ is in the saturation of the ideal generated by~$\Xi$.
  Hence,~$(f, \Xi) \in \ch_{d,g}(4)$.

  Conversely, let~$(f, \Xi) \in \ch_{d,g}(4)$, then~$R_{r-4} f \subset \Xi$
  and there is a full rank map~$\phi \colon \mathbb{C}^{q(r) - N_{r-4}} \to R_r$ such that $\mop{im} \phi$ complements~$R_{r-4} f$ in~$\Xi$.
  Furthermore, $\dim R_1 \Xi \leq q(r+1)$, because~$\Xi \in \ch_{d,g}$, so there is a full rank map~$\psi \colon R_{r+1} \to \mathbb{C}^{p(r+1)}$
  such that~$R_1 \Xi \subseteq \ker \psi$. So $(f, \Xi)$ is the image of~$(f, \phi, \psi) \in \mathcal{B}$.
\end{proof}

\begin{lemma}
  For any~$a \geq 0$, let~$\overline{\mathcal{B}}_a$ be the union of the codimension~$a$ components of~$\overline{\mathcal{B}}$.
  Then
  \[ \left[ \overline{\mathcal{B}}_a \right]_\mathbb{Z} \leq \left( 15 \log \left( d+2 \right) \eta + \theta_1 + \theta_2 + \theta_ 3 \right)^a \]
\end{lemma}

\begin{proof}
  Let~$\mathcal{B}'$ be the closed set defined by the constraints~\ref{item:1} and~\ref{item:1bis}.
  The constraints \ref{item:2} and \ref{item:3} are open, so any component of~$\overline{\mathcal{B}}$ is a component of~$\mathcal{B}'$.
  In particular~$ [ \overline{\mathcal{B}}_a ]_\mathbb{Z} \leq  [ \mathcal{B}'_a ]_\mathbb{Z}$.

  Constraint~\ref{item:1} is expressed with~$p(r+1) N_{r-3}$ polynomial equations
  of multidegree~$(1,0,1)$ (w.r.t.\ $f$, $\phi$ and~$\psi$ respectively). Namely, $\psi(mf) = 0$ for every monomial $m$ in~$R_{r-3}$. Each~$p(r+1)$ components of the
  equation~$\psi(mf) = 0$ involves a sum of 35~terms (since~$f$, as a quartic
  polynomial, contains only 35~terms) with coefficients~1. So the $1$-norm of
  these constraints is at most~$35$ (which is also at most~$N_r$, since~$r \geq 4$).

  Constraint~\ref{item:1bis} is expressed with~$4 p(r+1) (q(r) - N_{r-4})$ polynomial equations of multidegree~$(0,1,1)$.
  Namely, $\psi( v \phi(e)) = 0$ for any basis vector~$e$ and any variable~$v \in \left\{ w,x,y,z \right\}$.
  Each $p(r+1)$ component of the equation~$\psi( v \phi(e)) = 0$ involves a sum of~$N_r$ terms with coefficients~1.
  So the $1$-norm of these constraints is at most~$N_r$.

  The claim is then a consequence of Proposition~\ref{prop:bound-height-from-defining-equations},
  with~$s = p(r+1)N_{r-3} + 4 p(r+1) (q(r) - N_{r-4})$ and~$L = N_r$.
  We check routinely, with Mathematica, that~$s L \leq (d+2)^{15}$.
\end{proof}

\begin{theorem}\label{thm:nl-bounds}
  There is an absolute constant $A > 0$ such that 
  for any~$d > 0$ and~$g\geq 0$ we have
  \[ \deg(\pNL_{d,g}) \leq {A^{d^9}} \text{ and } \|\pNL_{d,g}\|_1 \leq {2^{A^{d^9}}}. \]
\end{theorem}

\begin{proof}
  We assume $\NL_{d,g}$ is non-emtpy, since these inequalities are trivially satisfied if $\NL_{d,g} = \emptyset$ with $\pNL_{d,g}=1$. Let $P_2 \eqdef \mathbb{P} \left( \mop{End}( \mathbb{C}^{q(r)-N_{r-4}}, R_r) \right)$ and $P_3 \eqdef \mathbb{P} \left( \mop{End}( R_{r+1}, \mathbb{C}^{p(r+1)}) \right)$ denote the second and third factors of $\ca$. Let~$\alpha \eqdef (q(r) - N_{r-4}) N_r - 1$ and~$\beta \eqdef p(r+1) N_{r+1} - 1$ denote the dimensions of~$P_2$ and $P_3$ respectively. 
  Let~$\mathcal{E}$ be the projection of~$\overline{\mathcal{B}}$ on~$\mathbb{P}(R_4) \times P_2$.
  The fibers of the map~$\overline{\mathcal{B}}\to \mathcal{E}$ are projective subspaces of~$P_3$ since Constraints~\ref{item:1} and~\ref{item:1bis} are linear in~$\psi$. The dimension of these fibers are~$\beta' \eqdef p(r+1)^2 - 1$.
  So, by Proposition~\ref{prop:projection_formula},
  \begin{equation}
    \theta_3^{\beta} [\mathcal{E}]_\mathbb{Z} \leq \theta_3^{\beta'} \left[ \overline{\mathcal{B}} \right]_\mathbb{Z}.
  \end{equation}
  Next, the map $\mathcal{B} \to \ch_{d,g}(4)$ factors through~$\mathcal{E}$ and the fibers of the corresponding map~$\mathcal{E} \to \ch_{d,g}(4)$ have dimension $\alpha'\eqdef (q(r)-N_{r-4}) q(r) - 1$.
  Finally, let~$e$ be the dimension of the fibers of the map~$\ch_{d,g}(4) \to \NL_{d,g}$. (If this dimension is not generically constant, we work one component at a time.)
  Once again, by Proposition~\ref{prop:projection_formula}, we obtain
  \begin{equation}
    \theta_2^{\alpha} [\NL_{d,g}]_\mathbb{Z} \leq \theta_2^{\alpha'+e} [\mathcal{E}]_\mathbb{Z}.
  \end{equation}
  Since~$[\NL_{d,g}]_\mathbb{Z} = m(\pNL_{d,g}) \eta + \deg(\pNL_{d,g}) \theta_1$, taking
  $L = 15 \log(d+2)$, we get
  \begin{align}
    \deg \pNL_{d,g} &\leq  \text{coeff of } {\theta_1 \theta_2^{\alpha - \alpha' - e} \theta_3^{\beta-\beta'}} \text{ in } \left( L\eta + \theta_1+\theta_2+\theta_3 \right)^{\alpha+\beta - \alpha' - \beta' - e + 1} \\
                   &\leq 3^{\alpha+\beta - \alpha' - \beta' - e + 1}.
  \end{align}
  The exponent is a polynomial in~$d$ and~$g$. Unless $d^2 \geq 8g-8$, $\NL_{d,g}$ is empty. So,
  we may bound the exponent with a polynomial only in~$d$, which turns out to be of degree~$9$.
  Therefore,~$\deg\pNL_{d,g} \leq A^{d^9}$ for some constant~$A > 0$.

  Similarly,
  \begin{align}
    m(\pNL_{d,g}) &\leq  \text{coeff of } {\eta \theta_2^{\alpha - \alpha' - e} \theta_3^{\beta-\beta'}} \text{ in } \left( L\eta + \theta_1+\theta_2+\theta_3 \right)^{\alpha+\beta - \alpha' - \beta' - e + 1} \\
                 &\leq (\alpha+\beta - \alpha' - \beta' - e + 1) L 3^{\alpha+\beta - \alpha' - \beta' - e} \\
                  &\leq 2^{O(d^9)}.
  \end{align}
  By \textcite[Lemma~2.30.3]{DAndreaKrickSombra_2013},
  \begin{equation}
    \|\pNL_{d,g}\|_1 \leq \exp(m(\pNL_{d,g})) 36^{\deg \pNL_\Delta},
  \end{equation}
  and this implies the claim, for some other constant~$A > 0$.
\end{proof}

\noindent For the following, we write $a \uparrow b$ for $a^b$. This is a right-associative operation.

\begin{corollary}\label{cor:bounds-NL-Delta}
  There is an absolute constant~$A > 0$ such that
  for any~$\Delta > 0$,
  \begin{equation*}\label{eq:NL_bounds_w_A}
  \deg(\pNL_\Delta) \leq A \uparrow \Delta \uparrow \tfrac 92 \text{ and } \|\pNL_\Delta\|_1 \leq 2 \uparrow A \uparrow \Delta \uparrow \tfrac92. 
  \end{equation*}
  In fact, one can obtain the following explicit bounds
\begin{equation*}\label{eq:NL_bounds}
  \deg(\pNL_\Delta) \leq 3^{(\Delta+20)^{9/2}} \text{ and } \log_2 \|\pNL_\Delta\|_1 \leq (\Delta+60)^5 3^{(\Delta+20)^{9/2}}.
\end{equation*}
\end{corollary}

\begin{proof}
  The first statement follows directly from~\eqref{eq:11} and Theorem~\ref{thm:nl-bounds} using a different~$A$.
 The second statement is found by carrying out the arguments in the proof of Theorem~\ref{thm:nl-bounds} with the help of a computer algebra system.
\end{proof}

\subsection{How good are these bounds?}  
We can compare our degree bounds for~$\pNL_\Delta$ to the exact degrees computed by~\textcite{MaulikPandharipande_2013},
from which it actually follows that
\begin{equation}\label{eq:18}
  \deg \pNL_\Delta = O(\Delta^\frac{19}{2}).
\end{equation}
This sharper bound does not directly imply a sharper bound on the height of~$\pNL_\Delta$ but it suggests the following conjecture.
This would improve subsequently Theorems~\ref{thm:separation-bound} and~\ref{thm:liouville}. In particular,  Equation~\eqref{eq:main_result} would be exponential in the size of the coefficients, as opposed to being doubly exponential.
\begin{conjecture}\label{conj:better-height-bounds}
  As $\Delta$ goes to $\infty$ we have
  \[ \log \|\pNL_\Delta\|_1\leq \Delta^{\frac{19}{2}+o(1)}. \]
\end{conjecture}

Now we turn to the details of \eqref{eq:18}. Following {\mancite\textcite{MaulikPandharipande_2013}} (but replacing~$q$ by~$q^8$),
consider the following power series
\begin{equation}
  A \eqdef \sum_{n \in \Z} q^{n^2}, \  B \eqdef \sum_{n\in \Z} (-1)^n q^{n^2},\ \Psi = 108 \sum_{n>0} q^{8n^2},
\end{equation}
and~$\Theta$ defined by
\begin{multline}
  2^{22} \Theta \eqdef 3A^{21}-81 A^{19}B^2 -627 A^{18}B^3 -14436 A^{17}B^4 -20007 A^{16}B^5\\
  -169092 A^{15}B^6 -120636 A^{14}B^7 -621558 A^{13}B^8 -292796 A^{12}B^9\\
  -1038366 A^{11}B^{10} -346122 A^{10}B^{11} -878388 A^{9} B^{12} -207186 A^8 B^{13}\\
  -361908 A^7 B^{14} -56364 A^6 B^{15} -60021 A^5 B^{16} -4812 A^4 B^{17} -1881 A^3 B^{18}\\ -27 A^2 B^{19}+ B^{21}.
\end{multline}
From~\parencite[Corollary~2]{MaulikPandharipande_2013}, we have, for any~$\Delta > 0$,
\begin{equation}\label{eq:30}
  \deg \pNL_\Delta \leq \text{coefficient of $q^\Delta$ in $\Theta - \Psi$}.
\end{equation}
In fact, this is an equality when the components of~$\NL_\Delta$ are given appropriate multiplicities.
Let~$\Theta[k]$ denote the coefficient of~$q^k$ in~$\Theta$.
By~\eqref{eq:30}, we only need to bound~$\Theta[\Delta]$ in order to bound~$\deg\pNL_\Delta$.
To do so, replace every negative sign in the definition of $\Theta$ by a positive sign, including those in $B$, to obtain the \emph{coefficientwise} inequality
\begin{equation}
  \Theta \le 6 \bigg( \sum_{n \in \mathbb{Z}} q^{n^2} \bigg)^{21}.
\end{equation}
The coefficient of $q^k$ in $\left( \sum_{n\in \mathbb{Z}} q^{n^2} \right)^{21}$ is
\begin{equation}\label{eq:set_k}
  r_{21}(k) \eqdef \#\left\{ (a_1,\dots,a_{21}) \in \Z^{21} \ \middle|\  \sum_{i} a_i^2 = k\right\}.
\end{equation}
The asymptotic bound~$r_{d}(k) = O(k^{\frac d2 -1})$, for~$d > 4$, is well known \parencite[e.g.][Satz~5.8]{Kraetzel_2000}.

\section{Separation bound}

We now state and prove the main results.
Recall that $a \uparrow b =  a^b$ is right associative and for $\gamma \in H_\Z$ we defined the discriminant $\Delta(\gamma)$ as $(\gamma\cdot h)^2 - 4 \gamma\cdot \gamma$.

\begin{theorem}\label{thm:separation-bound}
For any~$f \in U_4$ with algebraic coefficients  
there is a computable constant~$c > 1$ such that
for any $\gamma \in H^2(\Xf, \mathbb{Z})$,  if~$\gamma \cdot \omega_f \neq 0$, then
 \[ \left| \gamma \cdot \omega_f \right| > \left( 2 \uparrow c \uparrow \Delta(\gamma) \uparrow \tfrac92 \right)^{-1}  . \]
\end{theorem}

To make the connection with~\eqref{eq:intro},
recall the map~$T$ introduced in~\eqref{eq:tube}. %
We choose a basis~$\gamma_1,\dotsc,\gamma_{21}$ of~$H_3(\mathbb{P}^3\setminus X_f, \Z) \simeq H_\Z/ \Z h$,
write~$T(\gamma) = \sum_i x_i \gamma_i$
and observe that~$\Delta(\gamma)$ is a quadratic function of the coordinates~$x_i$,
so that~$\Delta(\gamma)^{\frac12} \leq C \max_i |x_i|$ for some constant~$C$ depending on the choice of basis.

\subsection{Multiplicity of Noether--Lefschetz loci}

The multiplicity of some nonzero polynomial~$F \in \mathbb{C}[x_1,\dotsc,x_s]$
at a point~$p \in \mathbb{C}^s$ is the unique integer~$k$ such that all partial
derivatives of~$F$ of order~$< k$ vanish at~$p$ and some partial derivative of
order~$k$ does not. It is denoted by~$\mop{mult}_p F$.

The multiplicity of~$\pNL_\Delta$ at some~$f\in U_4$ is related to the elements of~$\Pic(\Xf)$ with discriminant~$\Delta$.
For~$\Delta > 0$, let~$E_\Delta$ be a set of representatives of the equivalence classes of the relation~$\sim$ on~$H_\mathbb{Z}$ defined by
\begin{equation}
  \gamma \sim \gamma' \text{ if } \exists a\in \mathbb{Q}^*, b\in \mathbb{Q}: \gamma' = a \gamma + bh.
\end{equation}

\begin{lemma}
  \label{lem:mult-NL}
  For any~$f\in U_4$ and any~$\Delta > 0$,
  \[ \mult_f \pNL_\Delta = \# \left( \Pic \Xf \cap E_\Delta \right). \]
\end{lemma}

\begin{proof}
    Let~$\wNL_{\Delta}$ be the lift of~$\NL_{\Delta}$ in~$\widetilde{U}_4$. Arguing as in \S\ref{sec:defNL}, $\wNL_\Delta$ is the union of smooth analytic hypersurfaces:
    \begin{equation}
      \wNL_{\Delta} = \bigcup_{\eta \in E_{\Delta}} \mathcal{P}^{-1} \left\{ w\in \mathcal{D} \st w\cdot \eta = 0  \right\}.
    \end{equation}
    Then the same holds locally for $\NL_\Delta$.

  For any $f \in U_4$ it follows from the smoothness of branches of $\NL_\Delta$ that
  $\mop{mult}_f \pNL_{\Delta}$ is exactly the number of branches meeting at~$f$.
  The branches meeting at~$f$ are described by the elements of~$\Pic \Xf$ with discriminant~$\Delta$.
  Two elements~$\gamma$ and~$\gamma'$ describe the same branch (that is the same hyperplane section of~$\mathcal{D}$)
  if and only if~$\gamma' \sim \gamma$.
  So $\mop{mult}_f \pNL_\Delta$ is exactly the number of equivalence classes in $\left\{ \gamma \in \Pic \Xf \st \Delta(\gamma) = \Delta\right\}$ for this relation.
\end{proof}

\subsection{Proof of Theorem~\ref{thm:separation-bound}}

We first apply Corollary~\ref{coro:picard-perturbation}. 
Let~$\epsilon = 4\abs{\gamma\cdot \omega_f}$.
The corollary gives constants~$\Cf > 0$ and~$\epsf > 0$ (depending only on~$f$)
such that if~$\epsilon < \epsf$ then
there exists a monomial~$m \in R_4$
and $t\in \mathbb{C}$ such that
\begin{equation}
  \abs{t} \leq \Cf \epsilon\label{eq:15} 
\end{equation}
and
\begin{equation}\label{eq:28}
  \gamma \in \Pic X_{f + tm}.
\end{equation}
Assume that~$\epsilon < \epsf$ and let~$t$ and~$m$ be as above.
As $u$ varies, the number~$\# \left( \Pic(X_{f+um}) \cap E_\Delta \right)$
has a strict local maximum at~$u=t$.
By Lemma~\ref{lem:mult-NL}, so does~$\mult_{f+um} \pNL_{\Delta(\gamma)}$.
In particular, there is some higher-order partial derivative of~$\pNL_\Delta$
which vanishes at~$f+tm$ but not at~$f+um$, for~$u$ close to but not equal to~$t$.
Let $\alpha \in \mathbb{N}^{35}$ be the multi-index for which
\begin{equation}
  P \eqdef \frac{1}{\alpha_1! \dotsb \alpha_{35}!} \frac{\partial^{|\alpha|} \pNL_\Delta}{\partial u^\alpha} \in \mathbb{Z} [u_1,\dotsc,u_{35}]
\end{equation}
is this derivative.
For a monomial~$u^\beta \eqdef u_1^{\beta_1} \dotsb u_{35}^{\beta_{35}}$ we have
\begin{equation}
  \frac{1}{\alpha_1! \dotsb \alpha_{35}!}\frac{\partial^{|\alpha|} u^\beta}{\partial u^\alpha} = \prod_{i=1}^{35} \binom{\beta_i}{\alpha_i} u^{\beta - \alpha}.
\end{equation}
Since~$\binom{\beta_i}{\alpha_i} \leq 2^{\beta_i}$, it follows that
\begin{equation}
 \left \| \frac{1}{\alpha_1! \dotsb \alpha_{35}!} \frac{\partial^{|\alpha|} \pNL_\Delta}{\partial u^\alpha} \right\|_1 \leq 2^{\deg \pNL_\Delta} \|\pNL_\Delta\|_1.
\end{equation}

Let~$Q \in \Qbar[s]$ be the polynomial~$Q(s) \eqdef P(f+sm)$.
By construction~$Q \neq 0$ and~$Q(t) = 0$.
Clearly~$\deg Q \leq \deg \pNL_\Delta$, and we check that
\begin{equation}
  \|Q\|_1 \leq \|P\|_1 \left( \|f\|_1 + 1 \right)^{\deg P}.
\end{equation}
and then
\begin{equation}
  \|Q\|_1 \leq  2^{\deg \pNL_\Delta} \|\pNL_\Delta\|_1 \left( \|f\|_1 + 1 \right)^{\deg \pNL_\Delta}.
\end{equation}
From Corollary~\ref{cor:bounds-NL-Delta}, we find a constant~$c$ depending only on~$f$ such that 
\begin{equation}\label{eq:Qc}
  \deg Q \leq c \uparrow \Delta \uparrow \tfrac 92 \text{ and } \|Q\|_1 \leq 2 \uparrow c \uparrow \Delta \uparrow \tfrac 92.
\end{equation}

We write~$Q =\sum_{i=0}^{\deg Q} q_i s^i$.
Let~$k$ be the smallest integer such that~$q_k \neq 0$.
Since~$Q(t) = 0$, it follows that
\begin{equation}
  \abs{q_k t^k} \leq \sum_{i=k+1}^{\deg Q} \abs{q_i t^i}.
\end{equation}
If~$\epsilon < C_f^{-1}$, we have~$\abs{t} < 1$, by~\eqref{eq:15},
and it follows that
\begin{equation}\label{eq:one}
  \abs{t} \geq \frac{\abs{q_k}}{\|Q\|_{1}}.
\end{equation}

Let $D \ge 1$ be the degree of the number field generated by the coefficients of $f$. Let $H>0$ be an upper bound for the absolute logarithmic Weil height for the coefficient vector of $f$~\parencite[p.77]{Waldschmidt}. Then $q_k$ is an algebraic number defined by a polynomial expression $\widetilde{q}_k(f)$ in the coefficients of $f$ with $\widetilde{q}_k$ having integer coefficients. Liouville's inequality~\parencite[Proposition 3.14]{Waldschmidt} gives
  \begin{equation}
    \abs{q_k} \ge \|\widetilde{q}_k\|_1^{-D+1} e^{-D H \deg \widetilde{q}_k}.
  \end{equation}
 It is easy to see that $\deg \widetilde{q}_k \le \deg \pNL_\Delta$ and $\|\widetilde{q}_k\|_1 \le 2^{\deg \pNL_\Delta} \|\pNL_\Delta\|_1 $, the latter can be bounded by $\|Q\|_1$. 

By~\eqref{eq:15}, this leads to
\begin{equation}
  \label{eq:13}
  \epsilon \geq \left( 2\uparrow c \uparrow \Delta \uparrow \tfrac 92 \right)^{-D(1+H)},
\end{equation}
for some other constant~$c$ depending only on~$f$. 
Recall that~\eqref{eq:13} holds with the assumption that~$\epsilon \leq \epsilon_f$
and~$\epsilon < C_f^{-1}$. However, we can choose~$c$ large enough so that 
the right-hand side of~\eqref{eq:13} is smaller than~$\epsilon_f$ and~$C_f^{-1}$. 
Then~\eqref{eq:13} holds unconditionally. Absorb the outer exponent of~\eqref{eq:13} into $c$ to conclude the proof of Theorem~\ref{thm:separation-bound}. \qed

\subsection{Numbers \emph{à la} Liouville}

Let~$(\theta_i)_{i\geq 0}$ be a sequence of positive integers such that~$\theta_i$ is a strict divisor of~$\theta_{i+1}$
for all~$i \geq 0$ (in particular~$\theta_i \geq 2^i$.)
Consider the number
\[ L_\theta \eqdef \sum_{i=0}^\infty \theta_i^{-1}. \]
As a corollary to the separation bound obtained in Theorem~\ref{thm:separation-bound},
the following result states that~$L_\theta$ is not a ratio of periods of quartic surfaces when~$\theta$ grows fast enough.

\begin{theorem}\label{thm:liouville}
  If~$\theta_{i+1} \geq 2 \uparrow 2 \uparrow \theta_i \uparrow 10$,
  for all~$i$ large enough,
  then~$L_\theta$ is not equal to~$\frac{\gamma_1 \cdot \omega_f}{\gamma_2 \cdot \omega_f}$
  for any~$\gamma_1, \gamma_2 \in H_\mathbb{Z}$ and any~$f\in U_4$ with algebraic coefficients.
\end{theorem}

\begin{proof}
  Let~$l_k = \sum_{i=0}^k \theta_i^{-1}$. Since~$\theta_{i}$
  divides~$\theta_{i+1}$, we can write~$l_k = \frac{u_k}{\theta_k}$ for some
  integer~$u_k$. And since the divisibility is strict, $\theta_i \geq 2^i$
  and~${u_k} \leq 2 {\theta_k}$.
  Moreover
  \begin{equation}
    0 < L_\theta - l_k \leq 2 \theta_{k+1}^{-1},
  \end{equation}
  using~${\theta_{k+i+1}} \geq 2^{i} {\theta_{k+1}}$, for any~$i \geq 0$.
  Assume now that~$L_\theta = \frac{\gamma_1 \cdot \omega_f}{\gamma_2 \cdot \omega_f}$
  for some~$\gamma_1,\gamma_2 \in H_\mathbb{Z}$ and some~$f \in U_4$ with rational coefficients.
  Then, with
  \begin{equation}
    \gamma_k \eqdef \theta_k \gamma_1 - u_k \gamma_2,
  \end{equation}
  we check that~$\Delta(\gamma_k) = O(\theta_k^2)$ and that
  \begin{equation}
    0 < \abs{\theta_k} \abs{\gamma_2\cdot \omega_f} (L_\theta-l_k) = \abs{\gamma_k \cdot \omega_f}  \leq C \frac{\theta_{k}}{\theta_{k+1}},
  \end{equation}
  for some constant~$C$.
  By Theorem~\ref{thm:separation-bound},
  we obtain therefore
  \begin{equation}
     ( 2 \uparrow c \uparrow \theta_k \uparrow 9 )^{-1} \leq C \frac{\theta_k}{\theta_{k+1}},
  \end{equation}
  for some constant~$c > 0$ which depends only on~$f$.
  This contradicts the assumption on the growth of~$\theta$.
\end{proof}

\subsection{Computational complexity}
\label{sec:comp-compl}

Given a polynomial~$f \in \Qbar[w,x,y,z]\cap U_4$ and a cohomology
class~$\gamma \in H^2(X_f, \mathbb{Z})$, we can decide
if~$\gamma \in \Pic(X_f)$ (that is~$\gamma\cdot \omega_f = 0$) as follows:
\begin{enumerate}[(a)]
  \item\label{step:precomp} Compute the constant~$c$ in Theorem~\ref{thm:separation-bound};
  \item\label{step:numcomp} Let  $\epsilon = \left( 2\uparrow c\uparrow \Delta(\gamma) \uparrow \frac92 \right)^{-1}$
  and compute an approximation~$s \in \mathbb{C}$ of the period $\gamma \cdot \omega_f$ such that~$\abs{s - \gamma\cdot \omega_f} < \frac12\epsilon$.
\end{enumerate}
Then $\gamma$ is in $\Pic(\Xf)$ if and only if~$\abs{s} < \frac12 \epsilon$.

Computing the Picard group itself is an interesting application of this procedure.
Algorithms for computing the Picard group of~$\Xf$, or even just the rank of it,
break the problem into two: a part gives larger and larger lattices inside $\Pic(\Xf)$ while the other part
gets finer and finer upper bounds on the rank of~$\Pic(\Xf)$ \parencite{Charles_2014,HassettKreschTschinkel_2013,PoonenTestaLuijk_2015}. The computation stops when
the two parts meet.
Approximations from the inside are based on finding sufficiently many elements of~$\Pic(\Xf)$.
So while deciding the membership of~$\gamma$ in~$\Pic(X_f)$ can be solved by computing~$\Pic(X_f)$ first,
it makes sense not to assume prior knowledge of the Picard group and to study the complexity of deciding membership as~$\Delta(\gamma) \to \infty$,
with~$f$ fixed.

Step~\ref{step:precomp} does not depend on~$\gamma$, so only the complexity of Step~\ref{step:numcomp} matters, that is
the numerical approximation of~$\gamma\cdot \omega_f$.
This approximation amounts to numerically solving a Picard--Fuchs differential equation \parencite{Sertoz_2019}
and the complexity is~$(\log \frac{1}{\epsilon})^{1+o(1)}$ \parencite{HAKMEM,Hoeven_2001,Mezzarobba_2010,Mezzarobba_2016}.
With the value of~$\epsilon$ in Step~\ref{step:numcomp}, we have a complexity bound of $\exp(\Delta(\gamma)^{O(1)})$ for deciding membership.

For the sake of comparison, we may speculate about an approach that would
decide the membership of~$\gamma$ in~$\Pic(X_f)$ by trying to construct an
explicit algebraic divisor on~$X_f$ whose cohomology class is equal to~$\gamma$.
It would certainly need to decide the existence of a point satisfying some
algebraic conditions in some Hilbert scheme~$\mathcal{H}_{d,g}$,
with~$d = O(\Delta^{\frac12})$ and~$g = O(\Delta)$ (see \S \ref{sec:defNL}).
Embedding~$\mathcal{H}_{d,g}$ (or some fibration over it, as we did in~\S
\ref{sec:equat-relat-hilb}) in some affine chart of a projective space of
dimension~$d^{O(1)}$ will lead to a complexity 
of~$\exp( \Delta(\gamma)^{O(1)} )$ for deciding membership in this way.

However, if Conjecture~\ref{conj:better-height-bounds} holds true, then
the complexity of the numerical approach for deciding membership would reduce to~$\Delta(\gamma)^{O(1)}$.

\section{Concluding remarks}\label{sec:conclusion}

\subsection{Going beyond quartic surfaces}

There are two directions in which the main result, Theorem~\ref{thm:separation-bound}, of this article can, in principle, be generalized beyond quartic surfaces.

In the first direction, our effective methods naturally extend to complete intersections in complete simplicial toric varieties, provided the complete intersection has a K3 type middle cohomology satisfying the integral Hodge conjecture. By this last condition, we mean that a single period should govern if a homology cycle is algebraic.
For instance, cubic fourfolds satisfy all of these conditions~(\cite{Voisin2012}). Of course, polarized K3 surfaces of degrees $2$, $6$, $8$ also work, in addition to the degree $4$ case covered here.

To generalize the result to this context, one needs to compute two ingredients. The height and degree bounds for the image of a Hilbert schemes, and the ``spread'' of the period map (as in \S\ref{sec:inverse_function}). Our use of effective Nullstellensatz to compute heights clearly extends. To compute the spread, we used the Griffiths--Dwork reduction which continues to work for complete intersections in compact simplicial toric varieties~(\cite{Batyrev1994, Dimca1995, Mavlyutov1999}).

The second direction one could generalize the result is to stick with surfaces in~$\ppp$ but to increase the degree. In this case, we do not know how to control the vanishing of individual period integrals. However, Lefschetz $(1,1)$-theorem can be used to relate algebraic cycles to the simultaneous vanishing of a vector of periods coming from all holomorphic forms. For instance, on quintic surfaces one can separate $4$-dimensional (holomorphic) period vectors from one another. The deduction of the separation bounds would be possible from a parallel discussion to the one provided here. 
This application would make it possible to prove our heuristic Picard group computations of surfaces \parencite{LairezSertoz_2019}.

It would also be highly desirable to be able to numerically verify arbitrary, non-linear, relations between periods of quartics. However, in order to generalize our approach to this set-up, one would need the integral Hodge conjecture on products of quartic surfaces.

\subsection{Closed formulae for the bounds}

It is possible to determine a closed formula, involving the height of~$f$, that bounds the constant $c$ in~Theorem~\ref{thm:separation-bound}. We removed the deduction of such a formula due to the excessive technical complexity it presents. In addition, the pursuit of a human readable bound gets us further and further from the optimal bounds. We envisioned using the constant~$c$ on computer calculations where an algorithmic deduction of $c$ is possible and preferable. We designed our proofs so that such an algorithm is explicit in the proofs. An implementation of this algorithm would be beneficial after the bounds for the heights of the Noether--Lefschetz loci are brought down significantly.

\subsection{Optimal bounds}

We conjectured by analogy (Conjecture~\ref{conj:better-height-bounds}) that our bounds for the height of the Noether--Lefschetz locus can be lowered by one level of exponentiation. One can be more optimistic based on the following observation. For many example quartics $X_f$, we determined the equations for the Hilbert scheme of lines over each pencil $X_{f+t m}$ for monomials $m$. Then, going through the algorithm in the proofs, we computed sharper separation bounds on these example quartics. On these examples, the separation bound was around $10^{ -60 }$. In other words, it was sufficient to deduce wether a homology cycle was the class of a line using only $60$ digits of precision. This suggests that for homology cycles of small discriminant, optimal separation bounds may be small enough to be used in practice. It would be interesting to see if generalizing the work of Maulik and Pandharipande from degrees to heights by using the modularity of arithmetic Chow rings~(\cite{Kudla2003}) would give close to optimal bounds.

\subsection{Analogies with related work}\label{sec:analogy}

Our construction bears a remote resemblance to the analytic subgroup theorem of \textcite{Wustholz1989} and the period theorem of \textcite{Masser1993}. The analytic subgroup theorem and its applications work with the exponential map $\exp_A \colon T_0 A \to A$ of a (principally polarized) abelian variety $A$ over $\Qbar \subset \C$. The periods of $A$ form a lattice $\Lambda \eqdef \ker \exp A$. Let $P = \exp_A^{-1} A(\Qbar)$ be the periods of all algebraic points on $A$.

The analytic subgroup theorem implies that $\Qbar$-linear relations between any set of elements $S \subset P$ are determined by abelian subvarieties of $B$: there is an abelian subvariety such that $T_0 B$ coincides with the span of $S$. Observe that the linear relations live on the domain of the transcendental map $\exp_A$ and are converted to an algebraic subvariety on the codomain. When $S = \{\gamma\} \subset \Lambda$, the Masser--Wüstholz period theorem bounds the degree of smallest $B$ whose tangent space contains $\gamma$ using the height of $A$ and the norm of $\gamma$.

In our work, we consider the space $U_4$ of smooth homogeneous quartic polynomials of degree $4$ and its universal cover $\widetilde{U}_4 \to U_4$. We then take the (transcendental) period map $\cp \colon \widetilde U_4 \to H_\C$. Note that the $\Z$-relations between periods are realized as linear subspaces of the period domain whereas the preimage of these linear spaces are the Noether--Lefschetz loci. These Noether--Lefschetz loci map to algebraic hypersurfaces on the space $U_4$.  

Superficially, the main difference of the two approaches is the direction of the naturally appearing transcendental maps that linearize relations between periods. However, the nature of the two transcendental maps appearing in both constructions also differ substantially. 

\printbibliography

\end{document}